\documentclass[11pt]{amsart}
\usepackage{lmodern}
\usepackage[T1]{fontenc}
\usepackage{microtype}
\usepackage{amssymb}
\usepackage{amsthm}
\usepackage{amscd}
\usepackage{hyperref}
\usepackage{cite}

\usepackage{tikz}
\usetikzlibrary{fit}
\usetikzlibrary{shapes.geometric, positioning}

\newtheorem{theorem}{Theorem}[section]
\newtheorem{lemma}[theorem]{Lemma}

\newtheorem{cor}[theorem]{Corollary}

\theoremstyle{definition}
\newtheorem{definition}[theorem]{Definition}

\newcommand{\sh}{\operatorname{sh}}

\newcommand{\pa}{\operatorname{part}}

\makeatletter

\def\dotminussym#1#2{%
  \setbox0=\hbox{$\m@th#1-$}%
  \kern.5\wd0%
  \hbox to 0pt{\hss\hbox{$\m@th#1-$}\hss}%
  \raise.6\ht0\hbox to 0pt{\hss$\m@th#1.$\hss}%
  \kern.5\wd0}

\mathchardef\mhyphen="2D

\allowdisplaybreaks[2]

\newcommand{\Disc}[2]{\ensuremath{\mathtt{Disc}_{#1}[#2]}}

\begin{document}

\title{$\sigma$-Algebras for Quasirandom Hypergraphs}
\author{Henry Towsner}
\date{\today}
\thanks{Partially supported by NSF grant DMS-1340666.}
\address {Department of Mathematics, University of Pennsylvania, 209 South 33rd Street, Philadelphia, PA 19104-6395, USA}
\email{htowsner@math.upenn.edu}
\urladdr{\url{http://www.math.upenn.edu/~htowsner}}

\begin{abstract}
We examine the correspondence between the various notions of quasirandomness for $k$-uniform hypergraphs and $\sigma$-algebras related to measurable hypergraphs.  This gives a uniform formulation of most of the notions of quasirandomness for dense hypergraphs which have been studied, with each notion of quasirandomness corresponding to a $\sigma$-algebra defined by a collection of subsets of $[1,k]$.

We associate each notion of quasirandomness $\mathcal{I}$ with a collection of hypergraphs, the $\mathcal{I}$-adapted hypergraphs, so that $G$ is quasirandom exactly when it contains roughly the correct number of copies of each $\mathcal{I}$-adapted hypergraph.  We then identify, for each $\mathcal{I}$, a particular $\mathcal{I}$-adapted hypergraph $M_k[\mathcal{I}]$ with the property that if $G$ contains roughly the correct number of copies of $M_k[\mathcal{I}]$ then $G$ is quasirandom in the sense of $\mathcal{I}$.  This generalizes recent results of Kohayakawa, Nagle, R\"odl, and Schacht; Conlon, H\`an, Person, and Schacht; and Lenz and Mubayi giving this result for some notions of quasirandomness.
\end{abstract}

\maketitle

\section{Introduction}

A sequence of graphs $G_n=(V_n,E_n)$ with $|V_n|\rightarrow\infty$ is \emph{$p$-quasirandom} if for any $U\subseteq V_n$,
\[\left|{U\choose 2}\cap E_n\right|=p{|U|\choose 2}+o(|V_n|^2).\]
This notion has been extensively studied, and many equivalent formulations of $p$-quasirandomness are known \cite{MR1054011,MR2825535,MR1940121,MR2278000,MR2488748,MR2591049,MR1645698,MR1988980,MR2080111,MR2676838,MR930498}.  In particular, Chung, Graham, and Wilson \cite{MR1054011} showed that $\{G_n\}$ is $p$-quasirandom exactly when
\[\lim_{n\rightarrow \infty}t_H(G_n)=p^{|F|}\]
for every finite graph $H=(W,F)$, where $t_H(G_n)$ is the probability that a randomly selected $\pi:H\rightarrow G_n$ is a homomorphism.  They further showed there is a single choice of $H$---the $4$-cycle $C_4$---so that when $\lim_{n\rightarrow\infty}\frac{|E_n|}{|{V_n\choose 2}|}=p$ and $\lim_{n\rightarrow\infty}t_{C_4}(G_n)=p^4$, the sequence $G_n$ is $p$-quasirandom.

We are interested in the generalization of this equivalence to hypergraphs.  A variety of notions of quasirandomness for hypergraphs have been proposed \cite{chung:MR1138430,MR2864651,chung:MR1068494,frankl:MR1884430}, and work of Lenz and Mubayi \cite{2012arXiv1208.5978L}, building on work by Chung \cite{chung:MR1138430}, shows that these notions are not linearly ordered.  Since there are many equivalent characterizations of quasirandomness for graphs, an ongoing area of study has been finding analogous results for hypergraphs \cite{MR2864654,MR2595699,MR2864650}.

Different notions of quasirandomness for hypergraphs correspond to counting different kinds of hypergraphs.  Kohayakawa, Nagle, R\"odl, and Schacht \cite{MR2595699} have shown that one notion of quasirandomness for hypergraphs, $\mathtt{Disc}_p$, implies that $\lim_{n\rightarrow\infty}t_H(G_n)$ converges to the correct (i.e., the expected value in an actual random hypergraph) value when $H$ is a \emph{linear} hypergraph---a graph where any two edges share at most one vertex.  Conlon, H\`an, Person, and Schacht \cite{MR2864650} have shown that there is a particular linear hypergraph $M$ such that when $\lim_{n\rightarrow\infty}t_{M}(G_n)$ converges to the right value, $G_n$ is $\mathtt{Disc}_p$.   These results were further extended by Lenz and Mubayi \cite{2012arXiv1208.4863L} to the case of $\mathtt{Expand}_p[\mathcal{I}]$ where $\mathcal{I}$ is a partition of $[k]$ and and we wish to count ``$\mathcal{I}$-linear'' hypergraphs.

In this paper we generalize these results to a wider family of notions of quasirandomness.  We will characterize all these notions of quasirandomness for $k$-uniform hypergraphs as instances of a family we call \Disc{p}{\mathcal{I}} where $\mathcal{I}$ is a collection of subsets of $[0,k-1]$ which is \emph{subset-free}---if $I\subseteq J$ with $I,J\in\mathcal{I}$ then $I=J$.  These notions have been previously studied in connection with the hypergraph generalization of Szemer\'edi's regularity lemma and related results like Szemer\'edi's Theorem and hypergraph removal \cite{MR2195580,MR2373376,tao07,henry12:_analy_approac_spars_hyper}.

These notions seem to completely capture the notion of quasirandomness for dense hypergraphs.  In particular, for suitable choices of $\mathcal{I}$ we obtain all the notions in \cite{2012arXiv1208.5978L}, as well as some additional ones.  We also encounter weak notions which can be treated in the same way; for instance, suitable choice of $\mathcal{I}$ gives the graphs in which every vertex has approximately the same degree.

Our main result is the following:
\begin{theorem}\label{thm:main_1}
  Let $\{G_n\}$ with each $G_n=(V_n,E_n)$ be a sequence of $k$-uniform hypergraphs with $\lim_{n\rightarrow\infty}\frac{|E_n|}{|{V_n\choose k}|}=p$.  The following are equivalent:
  \begin{itemize}
  \item $\{G_n\}$ is \Disc{p}{\mathcal{I}},
  \item For every $\mathcal{I}$-adapted hypergraph $H=(W,F)$, $\lim_{n\rightarrow\infty}t_H(G_n)$ converges to $p^{|F|}$,
  \item $\lim_{n\rightarrow\infty}t_{M_k[\mathcal{I}]}(G_n)$ converges to $p^{2^{\mathcal{I}}}$.
  \end{itemize}
\end{theorem}
We define the notion of a $\mathcal{I}$-adapted hypergraph and the particular $\mathcal{I}$-adapted hypergraph $M_k[\mathcal{I}]$ in Section \ref{sec:subsetfree}.
 
In the case where $\mathcal{I}$ is the set of singletons, this is given in \cite{MR2595699,MR2864650}; in the case where $\mathcal{I}$ is a partition, this is given in \cite{2012arXiv1208.5978L}; in the case where $\mathcal{I}={[0,k-1]\choose l}$ for $1<l<k$, this is conjectured in \cite{MR2864650}.

We also show that these notions are distinct for distinct $\mathcal{I}$; many cases of this were shown in \cite{2012arXiv1208.5978L}, and our method is essentially the one used there.

We will use analytic methods: we pass from a sequence of hypergraphs to an infinite measurable hypergraph which is a limit of this sequence.  This approach to graph theory has been well-studied in recent years \cite{borgs:MR2455626,lovasz:MR2306658,lovasz:MR2274085,diaconis:MR2463439,aldous:MR637937,hoover:arrays,kallenberg:MR2161313,austin08,hrushovski,goldbring:_approx_logic_measure,MR2964622,tao07} under various names, including \emph{graph limits} and \emph{ultraproducts}.  

 The relationship between quasirandomness notions and Szemer\'edi's Regularity lemma and its variants is well-understood \cite{MR1099576}.  In the limiting graph, the analog of regularity can be seen as follows: the graph itself is a measurable subset $E$ whose characteristic function $\chi_E$ is then an $L^\infty$ function.  The measurable sets of pairs, $\mathcal{B}_2$, are a $\sigma$-algebra.  There is a natural sub-$\sigma$-algebra $\mathcal{B}_{2,1}\subseteq\mathcal{B}_2$ generated by rectangles, so we can consider the projection $\mathbb{E}(\chi_E\mid\mathcal{B}_{2,1})$.  The projection plays the role of the decomposition into regular pieces in the usual regularity lemma.  The analytic view of hypergraph regularity has been studied in \cite{MR2815610,elek07,tao:MR2259060,goldbring:_approx_logic_measure}.  A strong form of regularity due to Tao \cite{MR2212136} is precisely equivalent to the existence of this projection.  Similarly, hypergraph regularity corresponds to the existence of a series of projections onto a chain of descending $\sigma$-algebras.

In this paper we identify notions of quasirandomness with the particular $\sigma$-algebras studied in \cite{MR2212136,henry12:_analy_approac_spars_hyper}.  Theorem \ref{thm:main_1} described above then follows in a uniform way.

In Section \ref{sec:subsetfree} we introduce the basic hypergraph notions we need---the subset-free collections $\mathcal{I}$, the $\mathcal{I}$-adapted hypergraphs, and the hypergraphs $M_k[\mathcal{I}]$.  In Section \ref{sec:analytic} we introduce our analytic framework, the setting of the \emph{graded probability space}, and describe the $\sigma$-algebras $\mathcal{B}_{k,\mathcal{I}}$; we then state the main results about the analytic setting.  In Section \ref{sec:ultraproducts} we describe the main theorem about ultraproducts we need in this paper; to the extent possible, we avoid the details of the construction itself.  In Section \ref{sec:finite} we finally describe our quasirandomness notions for sequences of finite hypergraphs; we relate these notions to properties of the corresponding ultraproducts and show how our main results in the finitary setting follow from our results in the analytic setting.  Sections \ref{sec:counting} and \ref{sec:seminorms} prove the main analytic results.  In Section \ref{sec:separating} we show that our notions of randomness are all distinct, in both the finitary and analytic settings.

The author is grateful to Alexandra Kjuchukova for pointing out a mistake in a previous version of this paper.

\section{Notation}\label{sec:notation}

\begin{definition}
We write $[k]$ for the set $\{0,1,\ldots,k-1\}$.

When $V$ is a set, we write ${V\choose k}$ for the set of subsets of $V$ of size exactly $k$.

  $(V,E)$ is a \emph{$k$-uniform hypergraph} if $V$ is a set and $E\subseteq{V\choose k}$.  If $k\leq l$, $(V_0,\ldots,V_{l-1},E)$ is an \emph{$l$-partite $k$-uniform hypergraph} if each $V_i$ is a set, the sets $V_0,\ldots,V_{l-1}$ are pairwise disjoint, and $E\subseteq\bigcup_{s\in{[l]\choose k}}\prod_{i\in S}V_i$.  In either case, we call an element of $E$ a \emph{hyperedge}.

In an $l$-partite hypergraph, for any $x\in \bigcup_{i< l}V_i$ we write $\pa(x)$ for the unique $i$ such that $x\in V_i$.  If $e\subseteq\bigcup_{i< l}V_i$ we write $\pa(e)=\{\pa(x)\mid x\in e\}$.
\end{definition}
In the definition of an $l$-partite hypergraph, the hyperedges are exactly sets $e\in{\bigcup_{i< l}V_i\choose k}$ such that for each $i<l$, $|e\cap V_i|\leq 1$.

As usual, a $2$-uniform hypergraph is called a \emph{graph} and the hyperedges of a graph are \emph{edges}.

Throughout this paper we use a slightly unconventional notation for tuples.  When $V$ is a finite set, a \emph{$V$-tuple from $G$} is a function $\vec x_V:V\rightarrow G$.  We write $G^V$ for the set of $V$-tuples.  If for each $v\in V$ we have designated an element $x_v\in G$, we write $\vec x_V$ for the tuple $\vec x_V(v)=x_v$.  Conversely, if we have specified a $V$-tuple $\vec x_V$, we often write $x_v$ for $\vec x_V(v)$.  When $V,W$ are disjoint sets, we write $\vec x_V\cup \vec x_W$ for the corresponding $V\cup W$-tuple.  (We will always assume $V$ and $W$ are disjoint when discussing $V\cup W$-tuples.)  When $I\subseteq V$ and $\vec x_V$ is a given $V$-tuple, we write $\vec x_I$ for the corresponding $I$-tuple: $\vec x_I(i)=\vec x_V(i)$ for $i\in I$.  When $B\subseteq G^{W\cup V}$, we will write $B(\vec a_W)$ for the slice $\{\vec x_V\mid \vec a_W\cup \vec x_V\in B\}$ corresponding to those coordinates.  

If $\pi:V\rightarrow W$ is injective, $\vec x_V$ is a tuple, and $I\subseteq W$, we abuse notation to write $\vec x_{\pi^{-1}(I)}$ for the $I$-tuple $\vec x_{\pi^{-1}(I)}(i)=\vec x_V(\pi^{-1}(i))$.

\section{Subset-Free Sets}\label{sec:subsetfree}

In order to parameterize our notions of randomness, we need the following:
\begin{definition}
Suppose $|V|=k$.  A collection of subsets $\mathcal{I}$ of $V$ is called \emph{subset-free on $V$} if $\mathcal{I}$ is non-empty and there are no $I,J\in\mathcal{I}$ with $I\subsetneq J$.
\end{definition}

Three families of subset-free collections are of particular interest for us:
\begin{itemize}
\item Any partition of $[k]$ is subset-free,
\item For any $n\leq k$, ${[k]\choose n}$ is subset-free,
\item If we take any $l<k$, the collection of subsets of $[k]$ of size $k-1$ containing $[l]$ is subset-free; by abuse of notation, we write $[l]\cap{[k]\choose k-1}$ for this collection.
\end{itemize}

Subset-free subsets of $V$ are in one-to-one correspondence with \emph{downsets} of subsets of $V$---collections which are closed under subset---and the ideas here could also be expressed, with minor notational differences, in terms of downsets (see \cite{tao07}, where this approach is taken).

\begin{definition}
If $\mathcal{I}$ and $\mathcal{J}$ are subset-free collections on $V$ then we say $\mathcal{I}\leq_s\mathcal{J}$ if for every $I\in\mathcal{I}$ there is a $J\in\mathcal{J}$ with $I\subseteq J$.  If $\mathcal{I}$ is a subset-free collection on $V$ and $\mathcal{J}$ is a subset-free collection on $W$ with $|V|=|W|$ then we say $\mathcal{I}\leq\mathcal{J}$ if there is some $\pi:W\rightarrow V$ so that for each $I\in\mathcal{I}$, there is a $J\in\mathcal{J}$ so that $I\subseteq\pi(J)$.
\end{definition}
The $s$ in $\leq_s$ stands for \emph{strong}.  Note that we can consider $\leq$ even when $V=W$.

\begin{definition}
  If $\mathcal{I}$ is a non-empty collection of subsets of $V$, we define $\mathcal{I}^{\#}\subseteq\mathcal{I}$ to be those $I\in\mathcal{I}$ such that no $J\in\mathcal{I}$ has $J\supsetneq I$.
\end{definition}
Clearly $\mathcal{I}^{\#}$ is subset-free.

\begin{definition}
  Let $F=(V,E)$ be a finite $k$-uniform hypergraph.  For any $e$, we define the \emph{shadow cast by $E$ on $e$}, $\sh_E(e)$, to be 
\[\{e\cap e'\mid e'\in E\setminus\{e\}\}^{\#}.\]

We say $F=(V,E)$ is \emph{$\mathcal{I}$-adapted} if there is an ordering $E=\{e_0,e_1,\ldots,e_{|E|-1}\}$ such that for each $i< |E|$, $\sh_{\{e_0,\ldots,e_{i-1}\}}(e_i)\leq\mathcal{I}$.  

When $F=(V_1,\ldots,V_k,E)$ is a finite $k$-partite $k$-uniform hypergraph and $\mathcal{I}$ is subset-free on $[k]$, we say $F$ is \emph{strongly $\mathcal{I}$-adapted} if for each $e\in E$, $\pa(\sh_E(e))\leq_s\mathcal{I}$.
\end{definition}

For example, recall that a graph is linear if for any two distinct edges $e\neq e'$, $|e\cap e'|\leq 1$ (that is, two edges have at most one point in common).  A graph is linear exactly when it is ${[k]\choose 1}$-adapted.  On the other hand, every $k$-uniform hypergraph is ${[k]\choose k-1}$-adapted.  When $\mathcal{I}$ is a partition, the $\mathcal{I}$-adapted hypergraphs are precisely the $\mathcal{I}$-linear hypergraphs of \cite{2012arXiv1208.4863L}.

For each subset-free $\mathcal{I}$ on $[k]$, we identify a $k$-uniform hypergraph $M_k[\mathcal{I}]$.
\begin{definition}
The vertices $V_k[\mathcal{I}]$ of $M_k[\mathcal{I}]$ are pairs $(j,\tau)$ where $\tau:\mathcal{I}\rightarrow\{0,1,\ast\}$ is a function with $\tau(I)=\ast$ iff $j\in I$.  The edges of $M_k[\mathcal{I}]$ are given by functions $\sigma:\mathcal{I}\rightarrow\{0,1\}$ where the corresponding edge $e_\sigma$ consists of those $(j,\tau)$ such that for each $I\in\mathcal{I}$, $\tau(I)\in\{\sigma(I),\ast\}$.
\end{definition}
Note that for each edge $e_\sigma$ and each $j<k$, there is exactly one $\tau$ such that $(j,\tau)\in e_\sigma$, so $M_k[\mathcal{I}]$ is $k$-partite.

This definition of $M_k[\mathcal{I}]$ is abstract, but will be convenient for our purposes since it is easy to work with formally.  See Figure \ref{fig:mfigure} for $M_3[\mathcal{I}]$ for some choices of $\mathcal{I}$.

\begin{figure}
\begin{tabular}{ccc}
$\mathcal{I}$&Picture of $\mathcal{I}$&$M_3[\mathcal{I}]$\\
$\{\{1,2\},\{2,3\}\}$&
  \begin{tikzpicture}[node distance=0.2in,
point/.style={draw,shape=circle,fill=black,minimum size=1.5mm,inner sep=0pt}]
\node[point](a){};\node[point,below of=a](b){};\node[point,below of=b](c){};
\draw[cap=round] ([xshift=-1.4mm] a.center) -- ([xshift=-1.4mm] b.center) arc (180:360:1.4mm) --([xshift=1.4mm] a.center) arc (0:180:1.4mm);
\draw[cap=round] ([xshift=-1.4mm] b.center) -- ([xshift=-1.4mm] c.center) arc (180:360:1.4mm) --([xshift=1.4mm] b.center) arc (0:180:1.4mm);
\end{tikzpicture}
&
  \begin{tikzpicture}[node distance=0.2in,
    point/.style={draw,shape=circle,fill=black,minimum size=1.5mm,inner sep=0pt}]
\node[point](b0){};
\node[left of=b0](x){};
\node[right of=b0](y){};
\node[point,above of=x](a0){};
\node[point,above of=y](a1){};
\node[point,below of=x](c0){};
\node[point,below of=y](c1){};
\draw[rounded corners=0.1cm,fill=red,opacity=0.5] ([xshift=1.4mm] a0.center) -- ([xshift=1.4mm,yshift=1.4mm] x.center) -- ([yshift=1.4mm]b0.center) arc (90:-90:1.4mm) -- ([xshift=1.4mm,yshift=-1.4mm] x.center) -- ([xshift=1.4mm] c0.center) arc (0:-180:1.4mm) -- ([xshift=-1.4mm] a0.center) arc (180:0:1.4mm);
\draw[rounded corners=0.1cm,fill=green,opacity=0.5] ([xshift=-1.4mm] a1.center) -- ([xshift=-1.4mm] b0.center) -- ([xshift=-1.4mm] c0.center) arc (180:360:1.4mm) -- ([xshift=1.4mm] b0.center) -- ([xshift=1.4mm] a1.center) arc (0:180:1.4mm);
\draw[rounded corners=0.1cm,fill=blue,opacity=0.5] ([xshift=-1.4mm] a0.center) -- ([xshift=-1.4mm] b0.center) -- ([xshift=-1.4mm] c1.center) arc (180:360:1.4mm) -- ([xshift=1.4mm] b0.center) -- ([xshift=1.4mm] a0.center) arc (0:180:1.4mm);
\draw[rounded corners=0.1cm,fill=yellow,opacity=0.5] ([xshift=-1.4mm] a1.center) -- ([xshift=-1.4mm,yshift=1.4mm] y.center) -- ([yshift=1.4mm]b0.center) arc (90:270:1.4mm) -- ([xshift=-1.4mm,yshift=-1.4mm] y.center) -- ([xshift=-1.4mm] c1.center) arc (180:360:1.4mm) -- ([xshift=1.4mm] a1.center) arc (0:180:1.4mm);
\end{tikzpicture}
\\\\
$\{\{1,2\},\{3\}\}$&
  \begin{tikzpicture}[node distance=0.2in,
point/.style={draw,shape=circle,fill=black,minimum size=1.5mm,inner sep=0pt}]
\node[point](a){};\node[point,below of=a](b){};\node[point,below of=b](c){};
\draw[cap=round] ([xshift=-1.4mm] a.center) -- ([xshift=-1.4mm] b.center) arc (180:360:1.4mm) --([xshift=1.4mm] a.center) arc (0:180:1.4mm);
\draw[cap=round] ([xshift=-1.4mm] c.center)  arc (180:360:1.4mm) arc (0:180:1.4mm);
\end{tikzpicture}
&
  \begin{tikzpicture}[node distance=0.2in,
    point/.style={draw,shape=circle,fill=black,minimum size=1.5mm,inner sep=0pt}]
\node[point](a0){};
\node[point,right of=a0](a1){};
\node[point, below of=a0](b0){};
\node[point,right of=b0](b1){};
\node[point,below of=b0](c0){};
\node[point,right of=c0](c1){};
\draw[rounded corners=0.1cm,fill=red,opacity=0.5] ([xshift=-1.4mm] a0.center) -- ([xshift=-1.4mm] b0.center) -- ([xshift=-1.4mm] c0.center) arc (180:360:1.4mm) -- ([xshift=1.4mm] b0.center) -- ([xshift=1.4mm] a0.center) arc (0:180:1.4mm);
\draw[rounded corners=0.1cm,fill=green,opacity=0.5] ([xshift=-1.4mm] a1.center) -- ([xshift=-1.4mm] b1.center) -- ([xshift=-1.4mm] c0.center) arc (180:360:1.4mm) -- ([xshift=1.4mm] b1.center) -- ([xshift=1.4mm] a1.center) arc (0:180:1.4mm);
\draw[rounded corners=0.1cm,fill=blue,opacity=0.5] ([xshift=-1.4mm] a0.center) -- ([xshift=-1.4mm] b0.center) -- ([xshift=-1.4mm] c1.center) arc (180:360:1.4mm) -- ([xshift=1.4mm] b0.center) -- ([xshift=1.4mm] a0.center) arc (0:180:1.4mm);
\draw[rounded corners=0.1cm,fill=yellow,opacity=0.5] ([xshift=-1.4mm] a1.center) -- ([xshift=-1.4mm] b1.center) -- ([xshift=-1.4mm] c1.center) arc (180:360:1.4mm) -- ([xshift=1.4mm] b1.center) -- ([xshift=1.4mm] a1.center) arc (0:180:1.4mm);
\end{tikzpicture}
\\\\
$\{\{1\},\{2\},\{3\}\}$&
  \begin{tikzpicture}[node distance=0.2in,
point/.style={draw,shape=circle,fill=black,minimum size=1.5mm,inner sep=0pt}]
\node[point](a){};\node[point,below of=a](b){};\node[point,below of=b](c){};
\draw[cap=round] ([xshift=1.4mm] a.center) arc (0:360:1.4mm);
\draw[cap=round] ([xshift=1.4mm] b.center) arc (0:360:1.4mm);
\draw[cap=round] ([xshift=1.4mm] c.center) arc (0:360:1.4mm);
\end{tikzpicture}
&
  \begin{tikzpicture}[node distance=0.2in,
    point/.style={draw,shape=circle,fill=black,minimum size=1.5mm,inner sep=0pt}]
\node[point](a0){};
\node[point,right of=a0](a1){};
\node[point,right of=a1](a2){};
\node[point,right of=a2](a3){};
\node[point, below of=a0](b0){};
\node[point,right of=b0](b1){};
\node[point,right of=b1](b2){};
\node[point,right of=b2](b3){};
\node[point,below of=b0](c0){};
\node[point,right of=c0](c1){};
\node[point,right of=c1](c2){};
\node[point,right of=c2](c3){};
\draw[fill=green,opacity=0.5] ([xshift=-1.4mm] a0.center) -- ([xshift=-1.4mm] c0.center) arc (180:360:1.4mm) -- ([xshift=1.4mm] a0.center) arc (0:180:1.4mm);
\draw[rounded corners=0.1cm,fill=red,opacity=0.5] ([xshift=-1.4mm] a1.center) -- ([xshift=-1.4mm] b1.center) -- ([xshift=-1.4mm] c0.center) arc (180:360:1.4mm) -- ([xshift=1.4mm] b1.center) -- ([xshift=1.4mm] a1.center) arc (0:180:1.4mm);
\draw[rounded corners=0.1cm,fill=blue,opacity=0.5] ([xshift=-1.4mm] a2.center) -- ([xshift=-1.7mm] b0.center) -- ([xshift=-1.4mm] c1.center) arc (180:360:1.4mm) -- ([xshift=1.4mm] b0.center) -- ([xshift=1.4mm] a2.center) arc (0:180:1.4mm);
\draw[rounded corners=0.1cm,fill=yellow,opacity=0.5] ([xshift=-1.4mm] a3.center) -- ([xshift=-1.4mm] b1.center) -- ([xshift=-1.4mm] c1.center) arc (180:360:1.4mm) -- ([xshift=1.4mm] b1.center) --([xshift=1.4mm] a3.center) arc (0:180:1.4mm);
\draw[rounded corners=0.1cm,fill=purple,opacity=0.5] ([xshift=-1.4mm] a0.center) -- ([xshift=-1.4mm] b2.center) -- ([xshift=-1.4mm] c2.center) arc (180:360:1.4mm) -- ([xshift=1.4mm] b2.center) --([xshift=1.4mm] a0.center) arc (0:180:1.4mm);
\draw[rounded corners=0.1cm,fill=orange,opacity=0.5] ([xshift=-1.4mm] a1.center) -- ([xshift=-1.4mm] b3.center) -- ([xshift=-1.4mm] c2.center) arc (180:360:1.4mm) -- ([xshift=1.4mm] b3.center) --([xshift=1.4mm] a1.center) arc (0:180:1.4mm);
\draw[rounded corners=0.1cm,fill=black,opacity=0.5] ([xshift=-1.4mm] a2.center) -- ([xshift=-1.4mm] b2.center) -- ([xshift=-1.4mm] c3.center) arc (180:360:1.4mm) -- ([xshift=1.4mm] b2.center) --([xshift=1.4mm] a2.center) arc (0:180:1.4mm);
\draw[rounded corners=0.1cm] ([xshift=-1.4mm] a3.center) -- ([xshift=-1.4mm] b3.center) -- ([xshift=-1.4mm] c3.center) arc (180:360:1.4mm) -- ([xshift=1.4mm] b3.center) --([xshift=1.4mm] a3.center) arc (0:180:1.4mm);
\end{tikzpicture}
\\
\end{tabular}
\caption{$M_3[\mathcal{I}]$}
\label{fig:mfigure}
\end{figure}
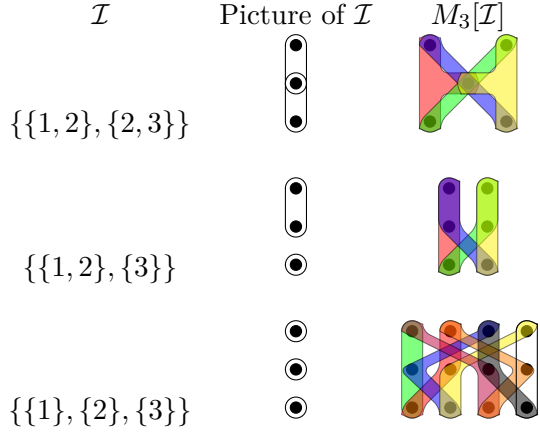

Many specific instances of the $M_k[\mathcal{I}]$ are familiar.  The most interesting graph case, $M_2[{[2]\choose 1}]$, is the cycle of length $4$.  More generally, $M_k[{[k]\choose k-1}]$ is the octahedron.  $M_k[{[k]\choose 1}]$ is the graph $M$ of \cite{MR2864650} and $M_k[[l]\cap{[k]\choose k-1}]$ is the squashed octahedron.

The hypergraphs $M_k[\mathcal{I}]$ can be described by successive doubling in the following way.  Observe that the vertices $V_k[\emptyset]$ are exactly pairs $(j,\langle\rangle)$, and there is exactly one edge of $M_k[\emptyset]$, namely $e_{\langle\rangle}=\{(j,\langle\rangle)\mid j<k\}$, so $M_k[\emptyset]$ consists of $k$ vertices with the only possible $k$-hyperedge.  Given $M_k[\mathcal{I}]$ and some new $I$ such that $\mathcal{I}\cup\{I\}$ is also subset-free, each vertex $(j,\tau)\in V_k[\mathcal{I}]$ corresponds to either one vertex $(j,\tau \ast)$ (when $j\in I$) or two vertices $(j,\tau 0)$ and $(j,\tau 1)$ (when $j\not\in I$).  Each edge $e_\sigma$ of $M_k[\mathcal{I}]$ corresponds to two edges of $M_k[\mathcal{I}\cup\{I\}]$, $e_{\sigma 0}$ and $e_{\sigma 1}$, with $e_{\sigma 0}\cap e_{\sigma 1}$ consisting of the $(j,\tau \ast)$ where $j\in I$.

\begin{lemma}
  $M_k[\mathcal{I}]$ is strongly $\mathcal{I}$-adapted.
\end{lemma}
\begin{proof}
Fix some edge $e_\sigma$ of $M_k[\mathcal{I}]$ where $\sigma:\mathcal{I}\rightarrow\{0,1\}$.  Define
\[\sigma_j(I)=\left\{
  \begin{array}{ll}
    \sigma(I)&\text{if }j\not\in I\\
    \ast&\text{if }j\in I
  \end{array}\right.\]
so $e_\sigma=\{(j,\sigma_j)\mid j<k\}$.

Consider any other edge of $M_k[\mathcal{I}]$; this other edge corresponds to a $\sigma':\mathcal{I}\rightarrow\{0,1\}$ which is not $\sigma$, so there is an $I\in\mathcal{I}$ with $\sigma'(I)\neq \sigma(I)$.  $e_{\sigma'}$ is $\{(j,\sigma'_j)\mid j<k\}$ where $\sigma'_j$ is defined similarly to $\sigma_j$.

In order for $(j,\sigma_j)$ to belong to $e_{\sigma'}$, we must have $\sigma_j=\sigma'_j$, and therefore $\sigma_j(I)=\ast$, which means $j\in I$.  Therefore any $(j,\sigma_j)\in e_\sigma\cap e_{\sigma'}$ must satisfy $j\in I$.  Since for any $\sigma'\neq \sigma$ there is such an $I$, we have shown that $\sh_{M_k[\mathcal{I}]}(e_\sigma)\leq_s\mathcal{I}$.  This holds for any edge $e_\sigma$, so $M_k[\mathcal{I}]$ is strongly $\mathcal{I}$-adapted.
\end{proof}

\section{The Analytic Setting}\label{sec:analytic}

\subsection{Graded Probability Spaces}
We will be working in the setting of a \emph{graded probability space} as introduced by Keisler \cite{Keisler1}.  Informally, a graded probability space is an ordinary probability measure space augmented by measures on $\sigma$-algebras of $n$-tuples, where these $\sigma$-algebras of $n$-tuples may be larger than those given by the usual product construction.

\begin{definition}
A \emph{graded probability space} is a structure $(\Omega,\{\mathcal{B}_k\}_{k\in\mathbb{N}},\{\mu^k\}_{k\in\mathbb{N}})$ such that:
\begin{itemize}
\item For each $k$, $(\Omega^k,\mathcal{B}_k,\mu^k)$ is a probability measure space,
\item If $B\in\mathcal{B}_k$ then for any permutation $\pi$ of $[k]$
\[B^\pi=\{\vec x_{\pi^{-1}([k])}\mid \vec x_{[k]}\in B\},\]
belongs to $\mathcal{B}_k$ and $\mu^k(B^\pi)=\mu^k(B)$,
\item $\mathcal{B}_n\times\mathcal{B}_m\subseteq\mathcal{B}_{n\times m}$,
\item For any $B\in\mathcal{B}_{n+m}$ and $\vec x_{[n]}\in\Omega^n$, $B(\vec x_{[n]})\in\Omega^m$,
\item For any $B\in\mathcal{B}_{n+m}$, the function $g(\vec x_{[n]})=\mu^m(B(\vec x_{[n]}))$ is $\mathcal{B}_n$-measurable and $\int g(\vec x_{[n]})d\mu(\vec x_{[n]})=\mu^{n+m}(B)$.
\end{itemize}
\end{definition}

If we take any probability measure space $(\Omega,\mathcal{B},\mu)$ and set $\mathcal{B}_k=\prod_{i<k}\mathcal{B}$ and $\mu^k=\prod_{i<k}\mu$ then $(\Omega,\{\mathcal{B}_k\},\{\mu^k\})$ is always a graded probability space.  We will mostly be concerned, however, with graded probability spaces where $\mathcal{B}_k$ has additional sets.

If $V$ is any finite set of size $l$, a graded probability space $(\Omega,\{\mathcal{B}_l\}_{k\in\mathbb{N}},\{\mu^l\}_{l\in\mathbb{N}})$ induces a measure space $(\Omega^V,\mathcal{B}_V,\mu^V)$ by taking any bijection $\pi:V\rightarrow[l]$ and identifying a tuple $\vec x_V$ with $\vec x_{[l]}(i)=\vec x_V(\pi^{-1}(i))$.  Since $\mathcal{B}_k,\mu^k$ are invariant under permutations, the choice of bijection does not matter.

We adopt the convention that, unless otherwise stated, when we take an integral with respect to $\mu^V$, the variables are $\vec x_V$.

If we are interested in evaluating a particular function $f$, however, the choice of bijection will matter unless $f$ is symmetric.  The main case will be the following: $(V,E)$ is a finite $k$-uniform hypergraph, $\vec x_V$ is an embedding of $V$ into $\Omega$, $e\in E$, $f\in L^\infty(\mathcal{B}_k)$, and we would like to interpret the value $f(\vec x_e)$.   In order for this notion to make sense, we need to turn $\vec x_e$ into an ordered tuple, since $f$ is a function on ordered $k$-tuples $x_0,\ldots,x_{k-1}$.

When we are considering a $k$-partite $k$-uniform hypergraph, there is a canonical choice of order: if $V=V_0\cup\cdots\cup V_{k-1}$ then each edge $e$ contains exactly one vertex $e_i$ from each $V_i$, and we write $f(\vec x_e)$ for $f(x_{e_0},\ldots,x_{e_{k-1}})$.  When $f$ is symmetric, the value of $f(\vec x_e)$ does not depend on the ordering of $e$, so taking any order $e=\{e_0,\ldots,e_{k-1}\}$ we may take $f(\vec x_e)$ to be $f(x_{e_0},\ldots,x_{e_{k-1}})$.  We will only write $f(\vec x_e)$ when either $e$ is an edge in a $k$-partite graph or $f$ is symmetric.

When $H=(V_0,\ldots,V_{k-1},E)$ is a finite $k$-partite $k$-uniform hypergraph, it makes sense to ask about the density of embeddings of $H$ into a measurable $A\subseteq\mathcal{B}_k$, or more generally, weighted embeddings given by an arbitrary function:
\begin{definition}
  We say $f\in L^\infty(\mathcal{B}_k)$ is \emph{symmetric} if for every permutation $\pi:[k]\rightarrow[k]$, $f\circ \pi=f$.  We say $A\in\mathcal{B}_k$ is symmetric if $\chi_A$ is.

Suppose that either $V=\bigcup_{i<k}V_i$, $H=(V_0,\ldots,V_{k-1},E)$ is a finite $k$-partite $k$-uniform hypergraph and $f\in L^\infty(\mathcal{B}_k)$, or $H=(V,E)$ and $f\in L^\infty(\mathcal{B}_k)$ is symmetric; then we define
\[t_H(f)=\int \prod_{e\in E}f(\vec x_e) d\mu^V.\]
When $A\in\mathcal{B}_k$, $t_H(A)=t_H(\chi_A)$.
\end{definition}
$t_H(A)$ is the probability that a randomly selected embedding of $V$ into $\Omega$ is a homomorphism.

\subsection{$\sigma$-Algebras}
We wish to work in certain sub-$\sigma$-algebras of $\mathcal{B}_k$; each of these $\sigma$-algebras corresponds to a notion of quasirandomness.

\begin{definition}
  Let $\mathcal{I}$ be subset-free on $[k]$.  $\mathcal{B}_{k,\mathcal{I}}$ is the sub-$\sigma$-algebra of $\mathcal{B}_k$ generated by sets of the form
\[B=\{\vec x_{[k]}\mid\text{for each }I\in\mathcal{I},\ \vec x_I\in B_I\}\]
where each $B_I\in\mathcal{B}_{I}$.

We write $\mathcal{B}_{k,n}$ for $\mathcal{B}_{k,{[k]\choose n}}$.
\end{definition}

For example $\mathcal{B}_{k,1}=\mathcal{B}_1\times\cdots\times\mathcal{B}_1$.  (Thus our need to consider graded probability spaces: if we only looked at the product, $\mathcal{B}_{k,1}$ would already contain all our sets; as we'll see, $\mathcal{B}_{k,1}$ is actually the smallest $\sigma$-algebra we want to consider.)

More generally, when $\mathcal{I}$ is a partition, $\mathcal{B}_{m,\mathcal{I}}$ is a product.  When $\mathcal{I}$ is not a partition, like ${[k]\choose 2}$, we need something more complicated than a product.  For instance, $\mathcal{B}_{3,2}$ is generated by sets of the form
\[\{(x,y,z)\mid (x,y)\in B_0, (x,z)\in B_1, (y,z)\in B_2\}.\]

Note that the $\mathcal{B}_{k,\mathcal{I}}$ will not be symmetric if $\mathcal{I}$ is not.  For instance, consider the partition $\mathcal{I}=\{\{0,1\},\{2\}\}$: $\mathcal{B}_{3,\mathcal{I}}=\mathcal{B}_2\times\mathcal{B}_1$, but if we close $\mathcal{B}_{3,\mathcal{I}}$ under permutations and then close it to a $\sigma$-algebra, we end up with $\mathcal{B}_{3,2}$.

(In general, the projection of a symmetric function onto $\mathcal{B}_{k,\mathcal{I}}$ is not really the right object to consider: we should instead consider the linear subspace of $L^2(\mathcal{B}_k)$ generated by closing $L^2(\mathcal{B}_{k,\mathcal{I}})$ under permutations and sums.  However the projection onto this larger space is constant if and only if the projection onto $\mathcal{B}_{k,\mathcal{I}}$ is constant, so in this paper we may safely ignore this complication.)

When $\bigcup\mathcal{I}$ is not $[k]$---that is, some $i<k$ is entirely missing from the $I\in\mathcal{I}$---this notion still makes sense.  For concreteness, consider the subset-free $\mathcal{I}=\{\{1\},\{2\}\}$ on $[3]$.  If $f$ is $\mathcal{B}_{3,\mathcal{I}}$-measurable then there is a function $f'$ which is $\mathcal{B}_{2,\mathcal{I}}$ measurable so that $f(\vec x_{[3]})=f'(\vec x_{[2]})$ for all $\vec x_{[3]}$.  Therefore we can usually reduce our arguments to the case where $\bigcup\mathcal{I}=[k]$.

\begin{lemma}
If $\mathcal{I}\leq_s\mathcal{J}$ then $B_{k,\mathcal{I}}\subseteq B_{k,\mathcal{J}}$.
\end{lemma}

Elements of some $\mathcal{B}_{k,\mathcal{I}}$ are in some sense ``structured'', and when $\mathcal{I}\leq_s\mathcal{J}$, elements of $\mathcal{B}_{k,\mathcal{I}}$ have a simpler structure than those in $\mathcal{B}_{k,\mathcal{J}}$.  Our view of randomness is the converse: a hypergraph $A$ is quasirandomness relative to $\mathcal{I}$ if $\mathcal{B}_{k,\mathcal{I}}$ provides no information about $A$.

Formally, we look at the expectation of $\chi_A$ with respect to the $\sigma$-algebra $\mathcal{B}_{k,\mathcal{I}}$.  We cannot expect this to be $0$, since the integral of the expectation must be the measure of $A$.  However we can ask for the expectation to be as simple as possible given that constraint: $A$ is random if its expectation is a constant function.

\begin{definition}
  We say $f\in L^2(\mathcal{B}_k)$ is \Disc{p}{\mathcal{I}} if $\mathbb{E}(f\mid\mathcal{B}_{k,\mathcal{I}})$ is the function constantly equal to $p$.  We say $E$ is \Disc{p}{\mathcal{I}} if $\chi_E$ is.

We write \Disc{p}{n} for \Disc{p}{{[k]\choose n}}.
\end{definition}

We view each \Disc{p}{\mathcal{I}} as a notion of randomness.  As we will describe in the next section, the known notions of quasirandomness can by recovered with suitable choices of $\mathcal{I}$.

We mention here one family of weak quasirandomness notions not discussed below: the case where $\mathcal{I}$ is a singleton.  When $\mathcal{I}$ is a singleton $I$, $M_k[\mathcal{I}]$ is the hypergraph consisting of two hyperedges which intersect on a set of size $I$.  It is easy to check that a hypergraph $A$ with $\mu^k(A)=p$ has $p^2$ copies of $M_k[\mathcal{I}]$ exactly if almost every set of size $|I|$ belongs to the correct number of edges (that is, for almost every $\vec x_I$, the measure of the set of $\vec x_{[k]\setminus I}$ such that $\vec x$ is an edge is $p$).

\subsection{Symmetry}

A hypergraph fails to be \Disc{p}{\mathcal{I}} if it correlates with a \emph{directed} hypergraph of the right kind.  Since our primary interest is symmetric subsets of $\Omega^k$, the notions of randomness we want to work with should actual focus on symmetric analogs of \Disc{p}{\mathcal{I}}.  Fortunately, these turn out to be equivalent: we will show that if $f$ fails to be \Disc{0}{\mathcal{I}} then there is a directed hypergraph measurable with respect to $\mathcal{B}_{k,\mathcal{I}}$ with the property that $f$ correlates with the underlying undirected hypergraph.

\begin{definition}
 Let $\mathcal{I}$ be subset-free on $W$ with $|W|=k$ and for each $I\in\mathcal{I}$, let $H_I$ be a symmetric element of $\mathcal{B}_I$.  Define
\[\mathcal{K}_k(\{H_I\}_{I\in\mathcal{I}})=\{\vec x_{[k]}\mid\exists \pi:[k]\rightarrow W\text{ a bijection such that }\forall I\in\mathcal{I}\ \vec x_{\pi^{-1}(I)}\in H_I\}.\]
\end{definition}
Clearly $\mathcal{K}_k(\{H_I\}_{I\in\mathcal{I}})$ is symmetric (i.e., a hypergraph).

\begin{lemma}\label{thm:symmetrizing}
Suppose $f\in L^\infty(\mathcal{B}_k)$ is symmetric and not \Disc{0}{\mathcal{I}}.  Additionally, suppose $(\Omega,\mathcal{B},\mu)$ is non-atomic.  Then there exist pairwise disjoint sets $\{S_i\}_{i<k}$ and for each $I\in\mathcal{I}$, $C_I\subseteq\prod_{i\in I}S_i$, so that if we take $H_I$ to be the closure of $C_I$ under permutations,
\[\left|\int_{\mathcal{K}_k(\{H_{I}\}_{I\in\mathcal{I}})} f\ d\mu^k\right|>0.\]
\end{lemma}
\begin{proof}
First we deal with the possibility that $\bigcup\mathcal{I}\neq[k]$.  Suppose there is some $i<k$ such that no $I\in\mathcal{I}$ has $i\in I$.  Then we could take $f'$ to be the projection of $f$ onto $\mathcal{B}_{[k]\setminus\{i\}}$, and since $f$ is not \Disc{0}{\mathcal{I}}, also $f'$ is not \Disc{0}{\mathcal{I}} (since $\mathcal{I}$ is also subset-free on $[k]\setminus\{i\}$).  Then we could apply the claim to $f'$, obtaining $H_I$ with the desired property, and these $H_I$ would also witness the statement for $f$.

So we may assume that for each $i<k$, there is some $I\in\mathcal{I}$ with $i\in I$.

Since $f$ is not $\Disc{0}{\mathcal{I}}$, we may fix $B_I\in\mathcal{B}_I$ so that, setting $B=\{\vec x_{[k]}\mid\text{for each }I\in\mathcal{I},\ \vec x_I\in B_I\}$, $\left|\int f\chi_Bd\mu^k\right|>0$.

  Since $(\Omega,\mathcal{B},\mu)$ is non-atomic, we may partition $\Omega=\bigcup_{j\leq D}R_j$ with $\mu(R_j)=1/D$ so $1/D$ is much smaller than $\left|\int f\prod_I\chi_{B_I}d\mu^k\right|$.  Define 
\[W=\{e\in\Omega^k\mid\forall j\leq d\,|e\cap R_j|\leq 1\}.\]
That is, $W$ is those edges which span exactly $k$ of the $R_j$.  By choosing $D$ sufficiently large, we can ensure that $\mu^k(W)$ is very close to $1$, and so
\[\left|\int f\chi_{B\cap W}d\mu^k\right|>0.\]

For any $x\in\Omega$, define $R(x)$ to be the unique $i\leq D$ with $x\in R_i$.  Consider any injective map $\rho:[k]\rightarrow[D]$.  Define $C_\rho$ to be those $k$-tuples $\vec x_{[k]}\in B$ such that $\rho(i)=R(x_i)$ for each $i<k$.  Observe that if $\vec x_{[k]}\in B\cap W$ then there is exactly one $\rho$ such that $\vec x_{[k]}\in C_\rho$.  Therefore there is some $\rho$ with
\[\left|\int f\chi_{C_\rho}d\mu^k\right|>0.\]

Fix such a $\rho$.  For any permutation $\pi:[k]\rightarrow[k]$, $C_{\rho\circ\pi}=C_\rho^\pi$, and by the injectivity of $\rho$, $C_{\rho\circ\pi}\cap C_{\rho\circ\pi'}=\emptyset$ when $\pi,\pi'$ are distinct permutations.  By the symmetry of $f$, $\int f\chi_{C_\rho}d\mu^k=\int f\chi_{C_{\rho\circ\pi}}d\mu^k$, so setting $C=\bigcup_\pi C_{\rho\circ\pi}$,
\[\left|\int f\chi_Cd\mu^k\right|>0.\]
It remains to choose sets $H_I$ so that $C=\mathcal{K}_k(\{H_I\}_{I\in\mathcal{I}})$.

We take $S_i=R_{\rho(i)}$, $C_I=B_I\cap\prod_{i\in I}S_i$, and let $H_I$ be the closure of $C_I$ under permutations.

Since $\mathcal{K}_k(\{H_I\}_{I\in\mathcal{I}})$ is closed under permutations, to show that $C\subseteq\mathcal{K}_k(\{H_I\}_{I\in\mathcal{I}})$ it suffices to show that $C_\rho\subseteq \mathcal{K}_k(\{H_I\}_{I\in\mathcal{I}})$.  If $\vec x_{[k]}\in C_\rho$ then for each $I\in\mathcal{I}$, take $\pi:I\rightarrow I$ to be the identity, so $x_i\in R_{\rho(\pi(i))}=S_i$ and $\vec x_I\in B_I$, and therefore $\vec x_I\in H_I$.  Therefore $\vec x_{[k]}\in\mathcal{K}_k(\{H_I\}_{I\in\mathcal{I}})$.  Therefore $C\subseteq\mathcal{K}_k(\{H_I\}_{I\in\mathcal{I}})$.

For the converse inclusion, consider some $\vec x_{[k]}\in\mathcal{K}_k(\{H_{I}\}_{I\in\mathcal{I}})$, so there is a $\pi:[k]\rightarrow[k]$ such that each $\vec x_{\pi^{-1}(I)}\in H_I$.  We claim that for each $i<k$, there is exactly one $i'$ such that $x_{i'}\in S_i$: if not, there would be an $i<k$ with no such $i'$.  But taking some $I\in\mathcal{I}$ with $i\in I$, it would not be possible to have $\vec x_{\pi^{-1}(I)}\in H_I$ in this case.  Therefore there is a permutation $\vec x'_{[k]}$ of $\vec x_{[k]}$ so that $x'_i\in S_i$ for each $i<k$ and $\vec x'_I\in H_I$ for each $I\in\mathcal{I}$.  Then $\vec x'_I\in B_I$, so $\vec x'_{[k]}\in C_\rho$ and therefore $\vec x_{[k]}\in C$.

Each $H_I$ is symmetric and $C=\mathcal{K}_k(\{H_I\}_{I\in\mathcal{I}})$, completing the proof.
\end{proof}

We are now in a position to state the infinitary versions of our main result:
\begin{theorem}\label{thm:main_inf}
Let $f\in L^\infty(\mathcal{B}_k)$ be symmetric and $\mathcal{I}$ subset-free on $[k]$.  Then the following are equivalent:
  \begin{enumerate}
  \item $f$ is \Disc{p}{\mathcal{I}},
  \item For every $\mathcal{I}$-adapted $k$-uniform hypergraph $H=(V,E)$, $t_H(f)=p^{|E|}$,
  \item $t_{M_k[\mathcal{I}]}(f)=p^{2^{|\mathcal{I}|}}$.
  \end{enumerate}
\end{theorem}

This will be proven in pieces: $(1)\Rightarrow (2)$ is Theorem \ref{thm:counting}, $(2)\Rightarrow (3)$ is immediate since $M_k[\mathcal{I}]$ is $\mathcal{I}$-adapted, and $(3)\Rightarrow (1)$ is Theorem \ref{thm:mk}.

\section{Ultraproducts}\label{sec:ultraproducts}

We relate the analytic setting described in the previous section to sequences of hypergraphs by using the ultraproduct construction.  Essentially, given a sequence of hypergraphs, this construction produces a graded probability space which is a limit of the given sequence.

Rather than reiterate that development here, we state a theorem which encapsulates all needed properties of the construction and refer the reader to \cite{goldbring:_approx_logic_measure} for a proof and a detailed exposition of the technique.
\begin{theorem}\label{thm:ultraproduct}
  Let $\{V_n\}$ be a sequence of finite sets with $|V_n|\rightarrow\infty$.  For any infinite set $S\subseteq\mathbb{N}$, there exist:
  \begin{itemize}
  \item A graded probability space $(\Omega,\{\mathcal{B}_k\},\{\mu^k\})$, and
  \item For 
every sequence of sets $\langle A_n\rangle$ with each $A_n\subseteq V_n^k$, a set $\lim\langle A_n\rangle=A\subseteq\Omega^k$ in $\mathcal{B}_k$,
  \end{itemize}
so that:
\begin{itemize}
\item Each $\mathcal{B}_k$ is generated by sets of the form $\lim\langle A_n\rangle$,
\item The operation $\lim$ commutes with union, intersection, and complement, so $\lim\langle A_n\cap B_n\rangle=\lim\langle A_n\rangle\cap\lim\langle B_n\rangle$ and similarly for $\cup$ and complement,
\item Given finitely many sequences $\langle A_{1,n}\rangle,\ldots,\langle A_{r,n}\rangle$ with each $A_{i,n}\subseteq V_n^k$ and setting $A_i=\lim\langle A_{i,n}\rangle$, there is a set $S'\subseteq S$ such that
\[\lim_{n\in S'}\frac{\left|\bigcap_{i\leq r}A_{i,n}\right|}{|V_n|^k}=\mu^k(\bigcap_{i\leq r}A_i).\]
\end{itemize}
\end{theorem}

We call any graded probability space $(\Omega,\{\mathcal{B}_k\},\{\mu^k\})$ together with the operation $\lim$ an \emph{ultraproduct} of the sequence $\{V_n\}$.  When we have specified a set $S$ in the theorem, we say the ultraproduct \emph{concentrates} on $S$.  The sets $\lim\langle A_n\rangle$ are called \emph{internal subsets} of $\Omega^k$.

Note that our ultraproducts are always non-atomic.

It follows from the theorem that if we have a sequence of uniformly presented simple functions $f_n:V_n^k\rightarrow\mathbb{R}$---that is, each $f_n$ has the form
\[f_n=\sum_{i\leq d}p_{i,n}\chi_{B_{i,n}}\]
where $d$ does not depend on $n$ and the $p_{i,n}$ are uniformly bounded---then there is a function $f=\lim\langle f_n\rangle$---an \emph{internal $L^\infty(\mathcal{B}_d)$ function}---such that given such functions $\langle f_{1,n}\rangle,\ldots,\langle f_{r,n}\rangle$ with $f_i=\lim\langle f_{i,n}\rangle$, there is a set $S'\subseteq S$ such that
\[\lim_{n\in S'}\frac{1}{|V_n|^k}\sum_{\vec x_{[k]}\in V_n^k}\prod_{i\leq r}f_{i,n}(\vec x_{[d]})=\int\prod_{i\leq r}f_i\,d\mu^k.\]
(Theorem 4.13 of \cite{goldbring:_approx_logic_measure} considers in more detail which sorts of functions we can expect this correspondence for.)


\section{Finite Consequences}\label{sec:finite}

\begin{definition}
If $H=(V,E)$ is a $k$-uniform hypergraph, $W$ a set, and $f:{W\choose k}\rightarrow\mathbb{R}$, we define $t_H(f)$ to be
\[\frac{1}{|W|^{|V|}}\sum_{\vec x_V\in W^V}\prod_{e\in E}f(\vec x_e).\]

When $F\subseteq{W\choose k}$, we write $t_H(F)$ for $t_H(\chi_F)$.
\end{definition}
This is equivalent to the usual definition of $t_H(F)$ as the fraction of functions from $V$ to $W$ which are homomorphisms into $F$.

We recall the main families of randomness notions from \cite{2012arXiv1208.5978L}.  We will phrase all our finitary notions of randomness in terms of limits of ratios to make the comparison with the infinite setting more explicit.
\begin{definition}
  Let $G_n=(V_n,E_n)$ be a sequence of $k$-uniform hypergraphs.

\begin{itemize}
\item $\{G_n\}$ is $\mathtt{Disc}_p$ if for each $\epsilon>0$ and any sequence of sets $U_n\subseteq V_n$ with $\frac{|U_n|}{|V_n|}\geq\epsilon$ for all $n$, 
\[\lim_{n\rightarrow\infty}\frac{\left|{U_n\choose k}\cap E_n\right|}{{|U_n|\choose k}}=p.\]

\item Let $k_1+\cdots+k_t=k$ and let $\mathcal{I}$ be the partition $\{[0,k_1-1],[k_1,k_1+k_2-1],\ldots,[k_1+\cdots+k_{t-1},k-1]\}$.  $\{G_n\}$ is $\mathtt{Expand}_p[\mathcal{I}]$ if for each $\epsilon>0$ and any sequence of sets $H_{i,n}\subseteq{V_n\choose k_i}$ with $\frac{|H_{i,n}|}{|{V_n\choose k_i}|}\geq\epsilon$ for all $i,n$, 
\[\lim_{n\rightarrow\infty}\frac{\left|\prod_{i=1}^t H_{i,n}\cap E_n\right|}{\prod_{i=1}^t|H_{i,n}|}=p.\]

\item For any $l$ with $1\leq l<k$, $\{G_n\}$ is $\mathtt{CliqueDisc}_p[l]$ if for each $\epsilon>0$ and any sequence of sets $B_n\subseteq {V_n\choose l}$ with $\frac{|\mathcal{K}_k(B_n)|}{|{V_n\choose k}|}\geq\epsilon$ for all $n$,
\[\lim_{n\rightarrow\infty}\frac{|\mathcal{K}_k(B_n)\cap E_n|}{|\mathcal{K}_k(B_n)|}=p\]
where $\mathcal{K}_k(B_n)$ is the set of $k$-tuples $e$ with ${e\choose l}\subseteq B_n$.

\item For any $l$ with $2\leq l\leq k$, $\{G_n\}$ is $\mathtt{Deviation}_p[l]$ if
\[\lim_{n\rightarrow\infty}t_{M_k[[l]\cap{[k]\choose k-1}]}(\chi_{E_n}-p)= 0.\]
\end{itemize}
\end{definition}
Note that $\mathtt{Disc}_p$ is exactly $\mathtt{CliqueDisc}_p[1]$.

In the definition of  $\mathtt{Deviation}_p[l]$, recall that $M_k[[l]\cap{[k]\choose k-1}]$ is the squashed octahedron.  We call a copy of the squashed octahedron in $V_n$ \emph{even} if an even number of edges belong to $E_n$ and \emph{odd} if an odd number of edges belong to $E_n$.  Then when $p=1/2$, $t_{M_k[[l]\cap{[k]\choose k-1}]}(\chi_{E_n}-1/2)$ counts the difference between the number of even and odd squashed octahedra, and approaches $0$ when this difference is small relative to the total number of squashed octahedra.

\begin{theorem}
  For any partition $\mathcal{I}$, $\{(V_n,E_n)\}$ is $\mathtt{Expand}_p[\mathcal{I}]$ if and only if in every ultraproduct of $\{V_n\}$, $\lim\langle E_n\rangle$ is $\mathtt{Disc}_p[\mathcal{I}]$.
\end{theorem}
\begin{proof}
  Let $\{(V_n,E_n)\}$ be given and suppose there is some ultraproduct of $\{V_n\}$, $(\Omega,\{\mathcal{B}_k\},\{\mu^k\})$ so that $E=\lim\langle E_n\rangle$ is not \Disc{p}{\mathcal{I}}.  If we do not have $\lim_{n\rightarrow\infty}\frac{|E_n|}{|V_n|}=p$ then certainly $\{(V_n,E_n)\}$ is not $\mathtt{Expand}_p[\mathcal{I}]$, so we have $\mu^k(E)=p$.

Since $E$ is not \Disc{p}{\mathcal{I}}, $\chi_E-p$ is not \Disc{0}{\mathcal{I}}, so by Lemma \ref{thm:symmetrizing} we may find symmetric sets $H_I\in\mathcal{B}_I$ so that 
\[\left|\int_{\mathcal{K}_k(\{H_I\}_{I\in\mathcal{I}})}\chi_E-p\, d\mu^k\right|>0.\]
Observe that $\prod_{I\in\mathcal{I}}H_I\subseteq\mathcal{K}_k(\{H_I\}_{I\in\mathcal{I}})$ and $\mathcal{K}_k(\{H_I\}_{I\in\mathcal{I}})$ is the closure of $\prod_{I\in\mathcal{I}}H_I$ under permutations.  Because the $H_I$ are supported on the pairwise disjoint sets $S_i$, any permutation of $[k]$ either permutes $\prod_{I\in\mathcal{I}}H_I$ to itself or to a disjoint set, and so by the symmetry of $E$,
\[\left|\mu^k(\prod_{I\in\mathcal{I}}H_I\cap E)-p\mu^k(\prod_{I\in\mathcal{I}}H_I)\right|=\left|\int_{\prod_{I\in\mathcal{I}}H_I}\chi_E-p\, d\mu^k\right|=\epsilon>0.\]
Without loss of generality, we may assume the sets $H_I$ are internal since the internal sets generate the $\sigma$-algebras $\mathcal{B}_I$, so $H_I=\lim\langle H_{I,n}\rangle$.  Necessarily we have $\mu^{|I|}(H_I)=\delta>0$ for some small enough $\delta$.

Therefore there is an infinite set $S'$ so that for every $n\in S'$, $\frac{|H_{I,n}|}{|{V_n\choose I}|}\geq\delta/2$ and
\[\left|\frac{|\prod H_{I,n}\cap E_n|}{|V_n|^k}-p\frac{|\prod H_{I,n}|}{|V_n|^k}\right|\geq\epsilon/2,\]
and since $\frac{|\prod H_{I,n}|}{|V_n|^k}\leq 1$,
\[\left|\frac{|\prod H_{I,n}\cap E_n|}{|\prod H_{I,n}|}-p\right|\geq\epsilon/2.\]
Therefore $\{G_n\}$ is not $\mathtt{Expand}_p[\mathcal{I}]$.

Conversely, if $\{G_n\}$ is not $\mathtt{Expand}_p[\mathcal{I}]$, we may find $\epsilon,\delta>0$, $H_{I,n}\subseteq{V_n\choose |I|}$ with each $\frac{|H_{I,n}|}{|{V_n\choose I}|}\geq\epsilon$, and an infinite set $S'$ so that for each $n\in S'$,
\[\left|\frac{|\prod_I H_{I,n}\cap E_n|}{|\prod_I H_{I,n}|}-p\right|\geq\delta.\]
In the ultraproduct concentrating on $S$, the set $\prod_I H_I\in\mathcal{B}_{k,\mathcal{I}}$ satisfies
\[\left|\int(\chi_E-p)\prod_I \chi_{H_I}d\mu^k\right|\geq\epsilon\delta,\]
so $\chi_E$ is not \Disc{p}{\mathcal{I}}.
\end{proof}

\begin{theorem}
For $l>1$, $\{(V_n,E_n)\}$ is $\mathtt{CliqueDisc}_p[l]$ if and only if in every ultraproduct of $\{V_n\}$, $\lim\langle E_n\rangle$ is $\mathtt{Disc}_p[l]$.
\end{theorem}
Note that the $l=1$ case is covered by the previous theorem.
\begin{proof}
Let $\{(V_n,E_n)\}$ be given and suppose there is some ultraproduct of $\{V_n\}$, $(\Omega,\{\mathcal{B}_k\},\{\mu^k\})$ so that $E=\lim\langle E_n\rangle$ is not \Disc{p}{l}.  Again, if we do not have $\lim_{n\rightarrow\infty}\frac{|E_n|}{|V_n|}=p$ then certainly $\{(V_n,E_n)\}$ is not $\mathtt{CliqueDisc}_p[l]$, so we have $\mu^k(E)=p$.

If $E$ is not \Disc{p}{l} then by Lemma \ref{thm:symmetrizing}, we may find sets $H_I$ for each $I\subseteq[k]$ with $|I|=l$ so that $\left|\int_{\mathcal{K}_k(\{H_I\}_{I\in\mathcal{I})}}\chi_E-p\,d\mu^k\right|>0$.
Since the elements of $\mathcal{I}$ have the same size, it makes sense to set $H=\bigcup_{I\in\mathcal{I}}H_I$.  Since $l>1$, $\mathcal{K}_k(\{H_I\}_{I\in\mathcal{I}})=\mathcal{K}_k(\{H\}_{I\in\mathcal{I}})$.

$H$ is arbitrarily well approximated by internal sets $H'$, so we may find an internal $H'$ with
\[\left|\mu^k(\mathcal{K}_k(\{H'\}_{I\in\mathcal{I}})\cap E)-p\mu^k(\mathcal{K}_k(\{H'\}_{I\in\mathcal{I}}))\right|=\left|\int_{\mathcal{K}_k(\{H'\}_{I\in\mathcal{I})}}\chi_E-p\,d\mu^k\right|=\delta>0.\]
Necessarily $\mu^k(\mathcal{K}_k(\{H\}_{I\in\mathcal{I}}))=\epsilon>0$.

Then $H'=\lim\langle H'_n\rangle$ and there is an infinite set $S'$ so that for $n\in S'$, $\frac{|\mathcal{K}_k(H'_n)|}{|{V_n\choose k}|}\geq\epsilon/2$ and 
\[\left|\frac{|\mathcal{K}_k(H'_n)\cap E_n|}{|V_n|^k}-p\frac{|\mathcal{K}_k(H'_n)|}{|V_n|^k}\right|\geq\delta/2\]
and therefore
\[\left|\frac{|\mathcal{K}_k(H'_n)\cap E_n|}{|\mathcal{K}_k(H'_n)|}-p\right|\geq\delta/2>0.\]
Therefore $\{G_n\}$ is not $\mathtt{CliqueDisc}_p[l]$.

For the converse, suppose $\{G_n\}$ is not $\mathtt{CliqueDisc}_p[l]$.  Then there is an $\epsilon>0$ and a sequence of sets $H_n\subseteq{V_n\choose l}$ with $\frac{|\mathcal{K}_k(H_n)|}{|{V_n\choose k}|}\geq\epsilon$ for all $n$ so that $\lim_{n\rightarrow\infty}\frac{|\mathcal{K}_k(H_n)\cap E_n|}{|\mathcal{K}_k(H_n)|}\neq p$.  There is an infinite set $S$ so that for all $n\in S$, the quantity is bounded away from $p$ by some $\delta>0$.  We take an ultraproduct concentrating on this set $S$, and we have
\[\left|\int(\chi_E-p)\mathcal{K}_k(H)d\mu^k\right|\geq\epsilon\delta>0.\]
Since $\mathcal{K}_k(H)$ is $\mathcal{B}_{k,l}$-measurable, $\chi_E-p$ is not \Disc{0}{l}, so $\chi_E$ is not \Disc{p}{l}.
\end{proof}


More generally, we define a finitary randomness notion for any subset-free $\mathcal{I}$.
\begin{definition}
Let $G_n=(V_n,E_n)$ be a sequence of $k$-uniform hypergraphs and let $\mathcal{I}$ be subset-free on $[k]$.  $\{G_n\}$ is $\mathtt{Disc}_p[\mathcal{I}]$ if for each $\epsilon>0$ and any sequences of sets $H_{I,n}\subseteq {V_n\choose l}$ with $\frac{|\mathcal{K}_k(\{H_{I,n}\}_{I\in\mathcal{I}})|}{|{V_n\choose k}|}\geq\epsilon$ for all $n$,
\[\lim_{n\rightarrow\infty}\frac{|\mathcal{K}_k(\{H_{I,n}\}_{I\in\mathcal{I}})\cap E_n|}{|\mathcal{K}_k(\{H_{I,n}\}_{I\in\mathcal{I}})|}=p.\]
\end{definition}

\begin{theorem}\label{thm:main_finite}
  For any $\{G_n\}=\{(V_n,E_n)\}$ and any subset-free $\mathcal{I}$ with $\lim_{n\rightarrow\infty}\frac{|E_n|}{|{V_n\choose k}|}=p$, the following are equivalent:
  \begin{enumerate}
  \item $\{G_n\}$ is $\mathtt{Disc}_p[\mathcal{I}]$,
  \item In every ultraproduct of $\{V_n\}$, $\lim\langle E_n\rangle$ is \Disc{p}{\mathcal{I}},
  \item For every fixed $k$-uniform $\mathcal{I}$-adapted $H=(W,F)$, $t_H(G_n)\rightarrow p^{|F|}$,
  \item $t_{M_k[\mathcal{I}]}(G_n)\rightarrow p^{2^{|\mathcal{I}|}}$.
  \end{enumerate}
\end{theorem}
\begin{proof}
$(1)\Rightarrow(2)$: Suppose there is an ultraproduct of $\{V_n\}$ where $E=\lim\langle E_n\rangle$ is not $\mathtt{Disc}_p[\mathcal{I}]$.  If $\mu^k(E)\neq p$ then certainly $\{G_n\}$ is not \Disc{p}{\mathcal{I}}, so we assume $\mu^k(E)=p$.

$\chi_E-p$ is not \Disc{0}{\mathcal{I}} so by Lemma \ref{thm:symmetrizing} there are symmetric $H_I\subseteq\mathcal{B}_I$ such that
\[\left|\int_{\mathcal{K}_k(\{H_I\}_{I\in\mathcal{I}})}\chi_E-p\,d\mu^k\right|>0.\]
The $H_I$ are approximated by internal sets, so we may replace the $H_I$ with internal sets $H'_I$ so that
\[\left|\int_{\mathcal{K}_k(\{H'_I\}_{I\in\mathcal{I}})}\chi_E-p\,d\mu^k\right|=\left|\mu^k(\mathcal{K}_k(\{H'_I\}_{I\in\mathcal{I}})\cap E)-p \mu^k(\mathcal{K}_k(\{H'_I\}_{I\in\mathcal{I}}))\right|=\epsilon>0.\]
We also have $\mu^k(\mathcal{K}_k(\{H'_I\}_{I\in\mathcal{I}}))=\delta>0$.

Then there must be an infinite sequence $S'$ so that for $n\in S'$, $\frac{|\mathcal{K}_k(\{H'_{I,n}\}_{I\in\mathcal{I}})|}{|{V_n\choose k}|}\geq\delta/2$ and 
\[\left|\frac{|\mathcal{K}_k(\{H'_{I,n}\}_{I\in\mathcal{I}})\cap E_n|}{|\mathcal{K}_k(\{H'_{I,n}\}_{I\in\mathcal{I}})|}-p\right|\geq\epsilon/2.\]
This contradicts the assumption that $\{G_n\}$ is $\mathtt{CliqueDisc}_p[\mathcal{I}]$.

$(2)\Rightarrow(1)$: Suppose $\{G_n\}$ is not $\mathtt{CliqueDisc}_p[\mathcal{I}]$.  Take an infinite sequence $S$ and sequences $H_{I,n}$ witnessing this, so for each $n\in S$, $\frac{|\mathcal{K}_k(\{H_{I,n}\}_{I\in\mathcal{I}})|}{|{V_n\choose k}|}\geq\epsilon$ and
\[\left|\frac{|\mathcal{K}_k(\{H_{I,n}\}_{I\in\mathcal{I}})\cap E_n|}{|\mathcal{K}_k(\{H_{I,n}\}_{I\in\mathcal{I}})|}-p\right|\geq\delta.\]
Then letting $H_I=\lim\langle H_{I,n}\rangle$, we have $\mu^k(\mathcal{K}_k(\{H_I\}))\geq\epsilon$ and
\[\left|\mu^k(\mathcal{K}_k(\{H_{I}\}_{I\in\mathcal{I}})\cap E)-p\mu^k(\mathcal{K}_k(\{H_{I}\}_{I\in\mathcal{I}}))\right|\geq\epsilon\delta>0.\]
Since $\mathcal{K}_k(\{H_I\}_{I\in\mathcal{I}})$ is $\mathcal{B}_{k,\mathcal{I}}$-measurable, it follows that $E$ is not \Disc{p}{\mathcal{I}}.

$(3)\Rightarrow (4)$ is immediate.

$(4)\Rightarrow (2)$: If $t_{M_k[\mathcal{I}]}(G_n)\rightarrow p^{2^{|\mathcal{I}|}}$ then in any ultraproduct, $t_{M_k[\mathcal{I}]}(E)=p^{2^{|\mathcal{I}|}}$.  By Theorem \ref{thm:main_inf}, $E$ is \Disc{p}{\mathcal{I}}.

$(2)\Rightarrow(3)$: Suppose $t_H(E_n)\not\rightarrow p^{|F|}$, so there is an infinite sequence $S$ with
\[\left|t_H(E_n)-p^{|F|}\right|\geq\epsilon\]
for all $n\in S$.  Then in the ultraproduct concentrating on $S$, $t_H(E)\neq p^{|F|}$, so by Theorem \ref{thm:main_inf}, $E$ is not \Disc{p}{\mathcal{I}}.
\end{proof}

\begin{cor}
$\{(V_n,E_n)\}$ is $\mathtt{Deviation}_p[l]$ if and only if in every ultraproduct of $\{V_n\}$, $\lim\langle E_n\rangle$ is \Disc{p}{[l]\cap{[k]\choose k-1}}.
\end{cor}

We conclude this section by noting that we can ``undo'' the ultraproduct construction by randomly sampling from the ultraproduct:
\begin{theorem}\label{thm:sampling}
  Let $(\Omega,\{\mu^k\},\{\mathcal{B}_k\})$ be a graded probability space and let $E$ be a symmetric subset of $\Omega^k$.  Then there is a sequence of finite $k$-uniform hypergraphs $\{(V_n,E_n)\}$ such that for every finite $k$-uniform hypergraph $H$, $\lim_{n\rightarrow\infty}t_H((V_n,E_n))=t_H(E)$.
\end{theorem}
\begin{proof}
We let $V_n$ consist of $n$ points in $\Omega$ chosen randomly and independently according to $\mu$, and set $E_n=E\upharpoonright{V_n\choose k}$.  For any fixed $H$, we claim that with probability $1$, $\lim_{n\rightarrow\infty}t_H((V_n,E_n))=t_H(E)$.  The claim then follows since there are countably many $H$, and therefore any sequence $\{(V_n,E_n)\}$ in the intersection of countably many sets of measure $1$ has the desired property.

Let $H=(W,F)$.  It suffices to show that for each $\epsilon>0$, when $n$ is sufficiently large and $V_n$ consists of $n$ points in $\Omega$ chosen at random, $|t_H((V_n,E_n))-t_H(E)|<\epsilon$ with probability $>1-\epsilon$.  For every map $\pi:W\rightarrow V_n$, define an indicator variable $X_\pi$ which is $1$ if $\pi$ is a homomorphism from $H$ to $(V_n,E_n)$ and $0$ otherwise.  Then $t_H((V_n,E_n))=\frac{1}{n^d}\sum_\pi X_\pi$.  Each $X_\pi$ is a Bernoulli random variable which is $1$ with probability $t_H(E)$, and the variance of $t_H((V_n,E_n))$ is
\[\frac{1}{n^{2d}}\left[\sum_\pi\mathrm{Var}(X_\pi)+2\sum_{\pi,\pi'}\mathrm{Cov}(X_\pi,X_{\pi'})\right].\]
Observe that $\mathrm{Cov}(X_\pi,X_{\pi'})$ is $0$ unless the ranges of $\pi$ and $\pi'$ intersect; since the number of pairs whose ranges intersect is on the order of $n^{2d-1}$, the $\frac{2}{n^{2d}}\sum_{\pi,\pi'}\mathrm{Cov}(X_\pi,X_{\pi'})$ term approaches $0$ as $n\rightarrow\infty$.  For each $\pi$, $\mathrm{Var}(X_\pi)=t_H(E)(1-t_H(E))$, so $Var(t_H((V_n,E_n)))\rightarrow 0$ as $n\rightarrow\infty$.  $\mathbb{E}(t_H((V_n,E_n)))=t_H(E)$, so by Chebyshev's inequality, 
\[\mathbb{E}(|t_H((V_n,E_n))-t_H(E)|>\epsilon)\leq \frac{\mathrm{Var}(t_H((V_n,E_n)))}{\epsilon^2}.\]
Choosing $n$ large enough, we make the right side side smaller than $\epsilon$ as needed.
\end{proof}

\section{Counting Subgraphs}\label{sec:counting}

\begin{theorem}\label{thm:counting}
Suppose $H=(V,E)$ is a $k$-uniform and $g\in L^\infty(\mathcal{B}_k)$ is \Disc{p}{\mathcal{I}}.  If either:
\begin{itemize}
\item $g$ is symmetric and $H$ is $\mathcal{I}$-adapted, or
\item $H$ is $k$-partite and strongly $\mathcal{I}$-adapted,
\end{itemize}
then $t_H(g)=p^{|E|}$.
\end{theorem}
\begin{proof}
By induction on $|E|$.  If $|E|=0$, this is trivial.  Otherwise, in the symmetric case, $H$ is $\mathcal{I}$-adapted and therefore has some edge $e_0\in E$ so that $\sh_E(e_0)\leq\mathcal{I}$ and $(V,E\setminus\{e_0\})$ is $\mathcal{I}$-adapted.  In the non-symmetric case, $H$ is strongly $\mathcal{I}$-adapted, and so has some edge $e_0$ so that $\sh_E(e_0)\leq_s\mathcal{I}$.
Observe that
  \begin{align*}
    \int\prod_{e\in E}g(\vec x_e)d\mu^V
&=\int g(\vec x_{e_0})\prod_{e\neq e_0}g(\vec x_e)d\mu^V\\
&=\int \int g(\vec x_{e_0})\prod_{e\neq e_0}g(\vec x_e) d\mu^ed\mu^{V\setminus e_0}.
  \end{align*}
  For each $\vec x_{V\setminus e_0}\in\Omega^{V\setminus e_0}$, consider the function $h_{\vec x_{V\setminus e_0}}(\vec x_{e_0})=\prod_{e\neq e_0}g(\vec x_e)=\prod_{e\neq e_0}g(\vec x_{e_0\cap e},\vec x_{e\setminus e_0})$.  For fixed $\vec x_{V\setminus e_0}$, this has the form $\prod_{e\neq e_0}g_{e,\vec x_{V\setminus e_0}}(\vec x_{e_0\cap e})$ for appropriate $g_{e,\vec x_{V\setminus e_0}}(\vec x_{e_0\cap e})$.  We may group these into functions $g_{I,\vec x_{V\setminus e_0}}(\vec x_{e_0\cap I})$ for each $I\in sh_E(e_0)$ so that
\[h_{\vec x_{V\setminus e_0}}(\vec x_{e_0})=\prod_{I\in sh_E(e_0)}g_{I,\vec x_{V\setminus e_0}}(\vec x_{e_0\cap I}).\]
This shows that $h_{\vec x_{V\setminus e_0}}(\vec x_{e_0})$ is $\mathcal{B}_{e_0,\sh_E(e_0)}$-measurable.

Since $g$ is \Disc{p}{\mathcal{I}} we have $\mathbb{E}(g\mid\mathcal{B}_{k,\mathcal{I}})=p$.  In the symmetric case, since $\sh_E(e_0)\leq\mathcal{I}$ and $g$ is symmetric, also $\mathbb{E}(g\mid\mathcal{B}_{k,\sh_E(e_0)})=p$.  In the non-symmetric case, $\sh_E(e_0)\leq_s\mathcal{I}$ so $\mathbb{E}(g\mid\mathcal{B}_{k,\sh_E(e_0)})=p$.

Therefore for each $\vec x_{V\setminus e_0}$,
\begin{align*}
  \int g(\vec x_{e_0})\prod_{e\neq e_0}g(\vec x_{e_0\cap e})\, d\mu^{e_0}
&=  \int g(\vec x_{e_0})h_{\vec x_{V\setminus e_0}}(\vec x_{e_0}) \, d\mu^{e_0}\\
&=  \int \mathbb{E}(g\mid\mathcal{B}_{e_0,\sh_E(e_0)})(\vec x_{e_0})h_{\vec x_{V\setminus e_0}}(\vec x_{e_0}) \, d\mu^{e_0}\\
&=  \int p\,h_{\vec x_{V\setminus e_0}}(\vec x_{e_0}) \, d\mu^{e_0}\\
&=  p\int \prod_{e\neq e_0}g(\vec x_{e_0}) \, d\mu^{e_0}.\\
\end{align*}
Integrating over $\vec x_{V\setminus e_0}$,
\begin{align*}
  \int \prod_{e\in E}g(\vec x_e)d\mu^V
&=\iint g(\vec x_{e_0})\prod_{e\neq e_0}g(\vec x_e) d\mu^ed\mu^{V\setminus e_0}\\
&=p\iint \prod_{e\neq e_0}g(\vec x_{ e}) d\mu^{e_0} d\mu^{V\setminus e_0}\\
&=p\int\prod_{e\neq e_0}g(\vec x_e)d\mu^V\\
&=p\,t_{(V,E\setminus\{e_0\})}(g)\\
&=p\cdot p^{|E|-1}\\
&=p^{|E|}.
\end{align*}
\end{proof}

\section{Seminorms}\label{sec:seminorms}

In this section we show that the hypergraphs $M_k[\mathcal{I}]$ characterize \Disc{p}{\mathcal{I}} hypergraphs, in the sense that whenever $t_{\mathcal{M}_k[\mathcal{I}]}(A)=p^{2^{|\mathcal{I}|}}$ where $p=\mu(A)$, $A$ is \Disc{p}{\mathcal{I}}.

\begin{definition}
  For $f\in L^\infty(\mathcal{B}_k)$, define
\[||f||_{M_k[\mathcal{I}]}=\left|t_{M_k[\mathcal{I}]}(f)\right|^{2^{-|\mathcal{I}|}}.\]

More generally, we define the corresponding inner product for $\{f_\sigma\}_{\sigma\in 2^{\mathcal{I}}}$ where each $f_\sigma\in L^\infty(\mathcal{B}_k)$ by
\[\langle \{f_\sigma\}_{\sigma\in 2^{\mathcal{I}}}\rangle_{M_k[\mathcal{I}]}=\int\prod_{\sigma}f_\sigma(\vec x_{e_\sigma}) d\mu^{V_k[\mathcal{I}]}.\]
\end{definition}
The norms $||\cdot||_{M_k[{k\choose k-1}]}$ are essentially the \emph{Gowers uniformity norms} \cite{gowers01,MR2373376}.

Note that since the graphs $M_k[\mathcal{I}]$ are $k$-partite, these notions are defined even when $f$ is not symmetric.


\begin{lemma}\label{thm:upperbound}
  \[\left|\langle \{f_\sigma\}_{\sigma\in 2^{\mathcal{I}}}\rangle\right|\leq\prod_\sigma ||f_\sigma||_{M_k[\mathcal{I}]}.\]
\end{lemma}
\begin{proof}
We consider some minimal $\mathcal{I}'\subseteq\mathcal{I}$ such that whenever $\sigma\upharpoonright\mathcal{I}'=\sigma'\upharpoonright\mathcal{I}'$, $f_\sigma=f_{\sigma'}$.  Such a $\mathcal{I}'$ certainly exists since $\mathcal{I}'=\mathcal{I}$ suffices.  We proceed by induction on $|\mathcal{I}'|$.

If $|\mathcal{I}'|=0$---that is, $f_\sigma=f_{\sigma'}$ for all $\sigma,\sigma'\in 2^{\mathcal{I}}$---this is trivial since, letting $f=f_\sigma$,
\[\left|\langle \{f\}_{\sigma\in 2^{\mathcal{I}}}\rangle\right|=\left|t_{M_k[\mathcal{I}]}(f)\right|=||f||_{M_k[\mathcal{I}]}^{2^{\mathcal{I}}}=\prod_\sigma||f||_{M_k[\mathcal{I}]}.\]

Suppose the claim holds for $|\mathcal{I}'|-1$ and $I\in\mathcal{I}'$; let us write $\mathcal{I}^-=\mathcal{I}\setminus\{I\}$.  Recall that for a vertex $(j,\tau)\in V_k[\mathcal{I}]$, there are three possibilities: $j\in I$ and $\tau(I)=\ast$, $j\not\in I$ and $\tau(I)=0$, or $j\not\in I$ and $\tau(I)=1$.  We partition $V_k[\mathcal{I}]=V_*\cup V_0\cup V_1$ according to which of these three cases holds.

For any $\sigma\in 2^{\mathcal{I}^-}$ and $b\in\{0,1\}$, we write $\sigma b$ for the function in $2^{\mathcal{I}}$ with $\sigma b(I')=\sigma(I')$ for $I'\in\mathcal{I}$ and $\sigma b(I)=b$.  Note that for each $\sigma$ and $b$, $e_{\sigma b}\subseteq V_*\cup V_b$.  Therefore
\begin{align*}
  |\langle \{f_{\sigma}\}_{\sigma\in 2^{\mathcal{I}}}\rangle_{M_k[\mathcal{I}]}|^{2^{|\mathcal{I}|}}
&=\left(\int \prod_{\sigma\in 2^{\mathcal{I}^-}}f_{\sigma 0}(\vec x_{e_{\sigma 0}})\prod_{\sigma\in 2^{\mathcal{I}^-}}f_{\sigma 1}(\vec x_{e_{\sigma 1}})d\mu^{V_k[\mathcal{I}]}\right)^{2^{|\mathcal{I}|}}\\
&=\left(\int \int \prod_{\sigma\in 2^{\mathcal{I}^-}}f_{\sigma 0}(\vec x_{e_{\sigma 0}})d\mu^{V_0}\cdot\right.\\
&\ \ \ \ \ \ \ \ \left.\int\prod_{\sigma\in 2^{\mathcal{I}^-}}f_{\sigma 1}(\vec x_{e_{\sigma 1}})d\mu^{V_1}d\mu^{V_*}\right)^{2^{|\mathcal{I}|}}\\
&\leq\left(\int \left(\int \prod_{\sigma\in 2^{\mathcal{I}^-}}f_{\sigma 0}(\vec x_{e_{\sigma 0}})d\mu^{V_0}\right)^2d\mu^{V_*}\cdot\right.\\
&\ \ \ \ \ \ \ \ \left.\int\left(\int\prod_{\sigma\in 2^{\mathcal{I}^-}}f_{\sigma 1}(\vec x_{e_{\sigma 1}})d\mu^{V_1}\right)^2d\mu^{V_*}\right)^{2^{|\mathcal{I}|-1}}\\
\end{align*}
with the final step following by Cauchy-Schwarz.

Observe that
\begin{align*}
  \int \left(\int \prod_{\sigma\in 2^{\mathcal{I}^-}}f_{\sigma 0}(\vec x_{e_{\sigma 0}})d\mu^{V_0}\right)^2d\mu^{V_*}
&=  \int \int \prod_{\sigma\in 2^{\mathcal{I}^-}}f_{\sigma 0}(\vec x_{e_{\sigma 0}})d\mu^{V_0}\cdot\\
&\ \ \ \ \ \ \ \ \int \prod_{\sigma\in 2^{\mathcal{I}^-}}f_{\sigma 0}(\vec x_{e_{\sigma 1}})d\mu^{V_1}d\mu^{V_*}\\
&=  \langle\{f_{(\sigma\upharpoonright \mathcal{I}^-)0}\}_{\sigma\in 2^{\mathcal{I}}}\rangle\\
&\leq\prod_{\sigma\in 2^{\mathcal{I}}}||f_{(\sigma\upharpoonright\mathcal{I}^-)0}||\\
&=\prod_{\sigma\in 2^{\mathcal{I}^-}}||f_{\sigma 0}||^2\\
\end{align*}
using the inductive hypothesis with $\mathcal{I}'\setminus\{I\}$, and similarly
\[\int \left(\int \prod_{\sigma\in 2^{\mathcal{I}^-}}f_{\sigma 1}(\vec x_{e_{\sigma 1}})d\mu^{V_1}\right)^2d\mu^{V_*}
\leq \prod_{\sigma\in 2^{\mathcal{I}^-}}||f_{\sigma1}||^2.\]

Therefore
\begin{align*}
  |\langle \{f_{\sigma}\}_{\sigma\in 2^{\mathcal{I}}}\rangle_{M_k[\mathcal{I}]}|^{2^{|\mathcal{I}|}}
&\leq\left(\prod_{\sigma\in 2^{\mathcal{I}^-}}||f_{\sigma0}||^2\prod_{\sigma\in 2^{\mathcal{I}^-}}||f_{\sigma1}||^2\right)^{2^{|\mathcal{I}|-1}}\\
&=\prod_{\sigma\in 2^{\mathcal{I}}}||f_{\sigma}||^{2^{|\mathcal{I}|}}.
\end{align*}
\end{proof}

\begin{cor}\label{thm:positivity}
  $\left|\int f\,d\mu^k\right|\leq||f||_{M_k[\mathcal{I}]}$.
\end{cor}
\begin{proof}
  Take $f_{\langle 0,\ldots,0\rangle}=f$ and $f_\sigma=1$ for all other $\sigma$ and apply the preceding lemma.
\end{proof}

\begin{lemma}
  $||\cdot||_{M_k[\mathcal{I}]}$ is a seminorm.
\end{lemma}
\begin{proof}
  Homogeneity is immediate, so we need only check subadditivity.  $||f+g||_{M_k[\mathcal{I}]}^{2^{|\mathcal{I}|}}$ expands to a sum of $2^{2^{|\mathcal{I}|}}$ integrals, each of which, by the previous lemma, is bounded by $||f||_{M_k[\mathcal{I}]}^m||g||_{M_k[\mathcal{I}]}^{2^{|\mathcal{I}|}-m}$ for a suitable $m$.  This sum is precisely $(||f||_{M_k[\mathcal{I}]}+||g||_{M_k[\mathcal{I}]})^{2^{|\mathcal{I}|}}$.
\end{proof}

\begin{lemma}
  Suppose $B$ is $\mathcal{B}_{k,I}$ measurable for some $I\in\mathcal{I}$.  Then for any $f\in L^\infty(\mathcal{B}_k)$,
\[||f\chi_B||_{M_k[\mathcal{I}]}\leq ||f||_{M_k[\mathcal{I}]}.\]
\end{lemma}
\begin{proof}
  We have $f=(f\chi_B)+(f\chi_{\overline{B}})$, so $||f||_{M_k[\mathcal{I}]}^{2^{|\mathcal{I}|}}$ expands to a sum of integrals of the form
\[\int\prod_{\sigma\in 2^{\mathcal{I}}}(f\chi_{S_\sigma})(\vec x_{e_\sigma})d\mu^{V_k[\mathcal{I}]}\]
where each $S_\sigma$ is either $B$ or $\overline{B}$.

Consider one of these terms.  Note that $\chi_{S_\sigma}$ only depends on coordinates in $I$, so $\chi_{S_\sigma}$ is a function of $I\cap e_\sigma$ (where $I$ is viewed as a subset of $V_k[\mathcal{I}]$ using the fact that $V_k[\mathcal{I}]$ is $k$-partite).  If there are any $\sigma,\sigma'$ which agree on $\mathcal{I}\setminus\{I\}$ but $S_\sigma\neq S_{\sigma'}$ then, since $I\cap e_\sigma=I\cap e_{\sigma'}$, for any $\vec x_{V_k[\mathcal{I}]}$, either $\chi_{S_\sigma}(\vec x_{V_k[\mathcal{I}]})=0$ or $\chi_{S_{\sigma'}}(\vec x_{V_k[\mathcal{I}]})=0$, and so the whole term is $0$.

Suppose not.  For each $\sigma\in 2^{\mathcal{I}}$, let $\hat\sigma=\sigma\upharpoonright(\mathcal{I}\setminus\{I\})$, so $\chi_{S_\sigma}$ depends only on $\hat\sigma$.  We partition $V_k[\mathcal{I}]=V_\ast\cup V_0\cup V_1$ depending on the value of $\tau(I)$.  So we have
\begin{align*}
\int\prod_{\sigma\in 2^{\mathcal{I}}}f\chi_{S_{\hat\sigma}}d\mu^{V_k[\mathcal{I}]}
&=\int \int\prod_{\hat\sigma\in 2^{\mathcal{I}\setminus\{I\}}}f\chi_{S_{\hat\sigma}}d\mu^{V_0}\cdot\\
&\ \ \ \ \ \ \ \ \int\prod_{\hat\sigma\in 2^{\mathcal{I}\setminus\{I\}}}f\chi_{S_{\hat\sigma}}d\mu^{V_1}d\mu^{V_*}\\
&=\int \left(\int\prod_{\hat\sigma\in 2^{\mathcal{I}\setminus\{I\}}}f\chi_{S_{\hat\sigma}}d\mu^{V_0}\right)^2d\mu^{V_*}\\
&\geq 0.
\end{align*}

Since all terms are nonnegative and one is $||f\chi_B||_{M_k[\mathcal{I}]}$, we have
\[||f\chi_B||_{M_k[\mathcal{I}]}\leq ||f||_{M_k[\mathcal{I}]}.\]
\end{proof}

\begin{theorem}\label{thm:main_for_0}
  $||f||_{M_k[\mathcal{I}]}=0$ iff $f$ is \Disc{0}{\mathcal{I}}.
\end{theorem}
\begin{proof}
  If $f$ is \Disc{0}{\mathcal{I}} then $||f||_{M_k[\mathcal{I}]}=t_{M_k[\mathcal{I}]}(f)^{2^{-|\mathcal{I}|}}=0$ by Theorem \ref{thm:counting} since $M_k[\mathcal{I}]$ is strongly $\mathcal{I}$-adapted.

Conversely, if $f$ is not \Disc{0}{\mathcal{I}} then there must exist sets $A_I\in\mathcal{B}_{k,I}$ so
\[\left|\int f\prod_{I\in\mathcal{I}}\chi_{A_I}d\mu^k\right|>0.\]
By Corollary \ref{thm:positivity}, $||f\prod_{I\in\mathcal{I}}\chi_{A_I}||_{M_k[\mathcal{I}]}>0$, and by the previous lemma, this means that $||f||_{M_k[\mathcal{I}]}>0$.
\end{proof}

\begin{lemma}\label{thm:upgrade}
  If $\int f d\mu^k=p$ and $||f||_{M_k[\mathcal{I}]}=p$ then for every $\mathcal{I}'\subseteq\mathcal{I}$, $||f||_{M_k[\mathcal{I}']}=p$.
\end{lemma}
\begin{proof}
It suffices to consider the case $\mathcal{I}=\mathcal{I}'\cup\{I\}$.  By Corollary \ref{thm:positivity}, $||f||_{M_k[\mathcal{I}']}\geq\mu^k(A)=p$, so suppose $||f||_{M_k[\mathcal{I}']}>p$.  Then
\begin{align*}
||f||_{M_k[\mathcal{I}]}^{2^{|\mathcal{I}'|}}
&=\int \prod_{\sigma\in 2^{\mathcal{I}}}f(\vec x_{e_{\sigma}})d\mu^{V_k[\mathcal{I}]}\\
&=\int \left(\int \prod_{\sigma\in 2^{\mathcal{I}'}}f(\vec x_{e_{\sigma}})d\mu^{V_0}\right)^2 d\mu^{V_*}\\
&\geq\left(\int \prod_{\sigma\in 2^{\mathcal{I}'}}f(\vec x_{e_{\sigma}})d\mu^{V_k[\mathcal{I}']}\right)^2\\
&=\left(||f||_{M_k[\mathcal{I}']}^{2^{|\mathcal{I}'|}}\right)^2\\
&>p^{2^{|\mathcal{I}|}}.
\end{align*}
\end{proof}

\begin{lemma}
  Suppose $f$ is \Disc{0}{\mathcal{I}'} for all $\mathcal{I}'\subsetneq\mathcal{I}$ and suppose that for every $\sigma\in 2^{\mathcal{I}}$, $h_\sigma\in\{1,f\}$.  Further, suppose there is at least one $\sigma$ with $h_\sigma=1$ and at least one $\sigma$ with $h_\sigma=f$.  Then
\[\langle \{h_\sigma\}_{\sigma\in 2^{\mathcal{I}}}\rangle_{M_k[\mathcal{I}]}=0.\]
\end{lemma}
\begin{proof}
  There must be some $I_0\in\mathcal{I}$ and some $\sigma^0\in 2^{\mathcal{I}}$ such that $h_{\sigma^0}=f$ and, letting $\sigma^1(I)=\sigma^0(I)$ for $I\neq I_0$ and $\sigma^1(I_0)=1-\sigma^0(I_0)$, $h_{\sigma^1}=1$.

  We recall some details about $M_k[\mathcal{I}]$: for any $\sigma$, $M_k[\mathcal{I}]$ contains an edge $e_\sigma$ whose vertices are $(j,\sigma_j)$ where $\sigma_j$ is defined by:
\[    \sigma_j(I)=\left\{\begin{array}{ll}
\sigma(I)&\text{if }j\not\in I\\
\ast&\text{if }j\in I.
\end{array}\right.\]

In particular, $e_{\sigma^0}\cap e_{\sigma^1}=\{(j,\sigma^0_j)\mid j\in I_0\}$.  We write $s_{I_0}=e_{\sigma^0}\cap e_{\sigma^1}$.  If $\tau\not\in\{\sigma^0,\sigma^1\}$ then $s_{I_0}\not\subseteq e_\tau$ (pick any $I\neq I_0$ with $\tau(I)\neq\sigma^0(I)$---one exists since $\tau\neq\sigma^0$ and $\tau\neq\sigma^1$, and pick any $j\in I_0\setminus I$; then $\sigma^0_j(I)\neq\tau_j(I)$, so $(j,\sigma^0_j)\neq(j,\tau_j)$).

Then
\begin{align*}
  \langle \{h_\sigma\}_{\sigma\in 2^{\mathcal{I}}}\rangle_{M_k[\mathcal{I}]}
&=\int\prod_{\sigma\in 2^{\mathcal{I}}}h_\sigma(\vec x_{e_\sigma}) d\mu^{V_k[\mathcal{I}]}\\
&=\int \int h_{\sigma^0}(\vec x_{e_{\sigma^0}}) d\mu^{e_{\sigma^0}}\prod_{\sigma\neq\sigma^0}h_\sigma(\vec x_{e_\sigma}) d\mu^{V_k[\mathcal{I}]\setminus e_{\sigma^0}}.\\
\end{align*}
For any fixed $\vec x_{V_k[\mathcal{I}]\setminus e_{\sigma^0}}$, note that the product $\prod_{\sigma\neq\sigma^0}h_\sigma(\vec x_{e_\sigma})$ is $\mathcal{B}_{k,\mathcal{I}\setminus\{I_0\}}$-measurable, and since $h_{\sigma^0}=f$ is \Disc{0}{\mathcal{I}\setminus\{I\}}, the inner integral is always $0$, so the whole integral is as well.
\end{proof}

\begin{theorem}\label{thm:mk}
  Suppose $||\chi_A||_{M_k[\mathcal{I}]}=p$ where $\mu(A)=p$.  Then $A$ is \Disc{p}{\mathcal{I}}.
\end{theorem}
\begin{proof}
  By induction on $|\mathcal{I}|$.  When $\mathcal{I}=\emptyset$ this is trivial because \Disc{p}{\emptyset} simply means that $\mu(A)=p$. Suppose the claim holds for each $\mathcal{I}'$ with $|\mathcal{I}'|<|\mathcal{I}|$ and let $A$ be given with $\mu(A)=p$ and $||\chi_A||_{M_k[\mathcal{I}]}=p$.  Then by Lemma \ref{thm:upgrade}, for each $\mathcal{I}'\subseteq \mathcal{I}$, $||\chi_A||_{M_k[\mathcal{I}']}=p$, and so by IH, $A$ is \Disc{p}{\mathcal{I}'} for all $\mathcal{I}'\subsetneq\mathcal{I}$.

We split $\chi_A$ into three components, $p$, $f=\mathbb{E}(\chi_A-p\mid\mathcal{B}_{k,\mathcal{I}})$, and $g=\chi_A-f-p$.  Then $g$ is \Disc{0}{\mathcal{I}} and by the inductive hypothesis, $f$ is \Disc{0}{\mathcal{I}'} for all $\mathcal{I}'\subsetneq\mathcal{I}$.  Then
\[p^{2^{|\mathcal{I}|}}=||\chi_A||_{M_k[\mathcal{I}]}^{2^{|\mathcal{I}|}}=||p+f+g||_{M_k[\mathcal{I}]}^{2^{|\mathcal{I}|}}.\]
The integral $||p+f+g||_{M_k[\mathcal{I}]}^{2^{|\mathcal{I}|}}$ expands into a sum of terms of the form
\[\langle \{h_\sigma\}_{\sigma\in 2^{\mathcal{I}}}\rangle_{M_k[\mathcal{I}]}\]
where each $h_\sigma\in\{p,f,g\}$.  By Theorem \ref{thm:main_for_0}, $||g||_{M_k[\mathcal{I}]}=0$, so a term with any $h_\sigma$ equal to $g$ must be $0$ by Lemma \ref{thm:upperbound}.  We can consider only the terms where each $h_\sigma\in\{p,f\}$.

Consider the mixed terms, where there is some $h_\sigma=p$ and some $h_{\sigma'}=f$.  By the preceding lemma, such terms must be $0$ (such a term is of the form $p^n\langle \{h'_\sigma\}\rangle$ where each $h'_\sigma\in\{1,f\})$.  Therefore only the unmixed terms remain:
\[p^{2^{|\mathcal{I}|}}=||\chi_A||_{M_k[\mathcal{I}]}^{2^{|\mathcal{I}|}}=||p||_{M_k[\mathcal{I}]}^{2^{|\mathcal{I}|}}+||f||_{M_k[\mathcal{I}]}^{2^{\mathcal{I}}}=p^{2^{|\mathcal{I}|}}+||f||_{M_k[\mathcal{I}]}^{2^{\mathcal{I}}}\]
and therefore $||f||_{M_k[\mathcal{I}]}^{2^{\mathcal{I}}}=0$.  By Theorem \ref{thm:main_for_0}, $f$ is \Disc{0}{\mathcal{I}}.  Since $f$ is also $\mathcal{B}_{k,\mathcal{I}}$-measurable, $f$ is $0$, and since $\mathbb{E}(\chi_A\mid\mathcal{B}_{k,\mathcal{I}})=p+f=p$, $\chi_A$ is \Disc{p}{\mathcal{I}}.
\end{proof}

\section{Separating Randomness Notions}\label{sec:separating}

In this section we show that the notions \Disc{p}{\mathcal{I}} are all distinct.  Our construction here is essentially same as the one used in Section 3 of \cite{2012arXiv1208.5978L}, where many particular cases of this result are shown.

\begin{lemma}
Let $(\Omega,\{\mathcal{B}_k\},\{\mu^k\})$ be an ultraproduct of the sets $\{V_n\}$ and for each $n$ let $<_n$ be a linear ordering of $V_n$.  Let $L_n=\{(x,y)\mid x<_ny\}$, $L=\lim\langle L_n\rangle$, and set $x<y$ iff $(x,y)\in L$.  Let $\pi$ be a permutation of $[k]$ and let $C_\pi$ be the set of $\vec x_{[k]}$ such that $x_i<x_j$ iff $\pi(i)<\pi(j)$.  Then $C_\pi$ is in $\mathcal{B}_{2,1}$.
\end{lemma}
\begin{proof}
  It suffices to show this with $\pi$ the identity, since any other $C_\pi$ is a permutation of this set.  So consider the set $C$ of $\vec x_{[k]}$ so that $x_i<x_j$ iff $i<j$.  We show that for every $\epsilon>0$, there is an element $C_\epsilon$ in $\mathcal{B}_{2,1}$ so that $\mu^k(C_\epsilon\bigtriangleup C)<1/\epsilon$.  Note that each interval $(x,y)=\{z\mid x<z<y\}$ is in $\mathcal{B}_1$ since $L$ is measurable and $(x,y)$ is an intersection of two slices of $L$.  Pick $m$ large enough and choose a partition $\Omega=\bigcup_{i< m}R_i$ where the $R_i$ are disjoint intervals with $\mu(R_i)=1/m$.  (Partitions approximating this exist in the $V_n$, so the limit of these partitions gives us the $R_i$.)  Order the $R_i$ so that $i<j$, $x\in R_i$, $y\in R_j$, implies $x<y$.  For each $x\in\Omega$, let $R(x)$ with the $i< m$ such that $x\in R_i$.

We define $C_m$ to consist of those $\vec x_{[k]}$ such that $i<j$ iff $R(x_i)<R(x_j)$.  Certainly $C_m\subseteq C$, and $\vec x_{[k]}\in C\setminus C_m$ iff there are $i\neq j$ with $R(x_i)=R(x_j)$.  The set of $\vec x_{[k]}$ where $i\neq j$ implies $R(x_i)\neq R(x_j)$ has measure $\frac{{m\choose k}k!}{m^k}$, so the complement has measure $1-\frac{{m\choose k}k!}{m^k}$.  By choosing $m$ large enough, we may arrange for $1-\frac{{m\choose k}k!}{m^k}<\epsilon$, and since $C\setminus C_m$ is contained in a set with this measure, $\mu^k(C\bigtriangleup C_m)<\epsilon$.

Since $C$ is a measurable set arbitrarily well approximated by elements of the $\sigma$-algebra $\mathcal{B}_{2,1}$, $C$ itself belongs to $\mathcal{B}_{2,1}$.
\end{proof}

\begin{theorem}
  Let $(\Omega,\{\mathcal{B}_k\},\{\mu^k\})$ be an ultraproduct.  For any $\mathcal{I}$ subset-free on $[k]$ and any $p\in(0,1)$, there is a symmetric $E\subseteq\Omega^k$ with $\mu^k(E)=p$ such that $E$ is not \Disc{p}{\mathcal{I}}, but for any $\mathcal{J}$ such that $\mathcal{I}\not\leq\mathcal{J}$, $E$ is \Disc{p}{\mathcal{J}}.
\end{theorem}
\begin{proof}
$(\Omega,\{\mathcal{B}_k\},\{\mu^k\})$ is an ultraproduct of graphs with underlying vertex sets $V_n$.  We fix a linear ordering $<_n$ on each $V_n$, and let $L_n=\{(x,y)\mid x<_n y\}$ and $L=\lim\langle L_n\rangle$.  We then take the corresponding linear ordering of $\Omega$, $x<y$ when $(x,y)\in L_n$.

First, suppose $\bigcup\mathcal{I}\neq[k]$.  Then $\mathcal{I}$ is subset-free on $[k]\setminus\{i\}$ for some $i<k$, so we could find $E\subseteq\Omega^{[k]\setminus\{i\}}$ witnessing the claim for $\mathcal{I}$ on $[k]\setminus\{i\}$, and then take $\{\vec x_{[k]}\mid \vec x_{[k]\setminus\{i\}}\in E\}$ to witness the claim for $\mathcal{I}$ on $[k]$.  Therefore we assume that $\bigcup\mathcal{I}=[k]$.

For every $I\in\mathcal{I}$ and every $n$, set $A^{\langle\rangle}_{I,n}=V_n$.  For every binary sequence $\sigma$, choose a random partition of ${V_n\choose |I|}$ into $A^{\sigma^\frown \langle0\rangle}_{I,n}$ and $A^{\sigma^\frown\langle 1\rangle}_{I,n}$.  Clearly for each $\sigma$, with probability $1$, $\lim_{n\rightarrow\infty}\frac{|A^\sigma_{I,n}|}{|V_n|}=2^{-|\sigma|}$ and the sequences $\langle A^{\sigma}_{I,n}\rangle$ are $\mathtt{CliqueDisc}_{2^{-|\sigma|}}[|I|-1]$.  Associate each sequence $\sigma$ with the dyadic rational $q(\sigma)=\sum_{i<|\sigma|}\sigma(i)2^{-i}$.  For every $p\in(0,1)$, we may choose sequences $m_{p,n}$ and $\sigma_{p,n}$ with $\lim_{n\rightarrow\infty}q(\sigma_{p,n})=p$ so that $\lim_{n\rightarrow\infty}\frac{|A^{\sigma_{p,n}}_{I,n}|}{|V_n|}=p$.  We define $A^p_I=\lim\langle A^{i_{p,n}}_{I,m_{p,n}}\rangle$, so $A^p_I$ is \Disc{p}{|I|-1}.  Note that when $p<q$ we have $A^p_I\subseteq A^q_I$.

Outside a set of measure $0$, any $\vec x_{[k]}$ is linearly ordered by $<$; given $\vec x_{[k]}$, let $\vec x^\dagger_{[k]}$ be the permutation of $\vec x_{[k]}$ so that $i<j$ implies $x^\dagger_i< x^\dagger_j$.  We define $f_I(\vec x_I)=\inf\{p\mid \vec x_I\in A^p_I\}$.  We take $E=\{\vec x_{[k]}\mid \sum_{I\in\mathcal{I}}f_I(\vec x^\dagger_{I})\mod 1<p\}$.  $E$ is symmetric since it depends only on $\vec x^\dagger_{[k]}$.

For each permutation $\pi$ of $[k]$, let $C_\pi$ be as in the preceding Lemma the set of $\vec x_{[k]}$ with $x_i<x_j$ iff $\pi(i)<\pi(j)$.  Then when $\pi\neq\pi$, $C_\pi\cap C_{\pi'}=\emptyset$ and for each $\pi$, $E\cap C_\pi$ is in $\mathcal{B}_{k,\mathcal{I}}$ since $E\cap C_\pi$ has the form $\{\vec x_{[k]}\mid\sum_{I\in\mathcal{I}}f_I(\vec x_{\pi^{-1}(I)})\}$ for a suitable choice of $\pi$.  Therefore the union $E=\bigcup_\pi E\cap C_\pi$ is $\mathcal{B}_{k,\mathcal{I}}$-measurable.  Since $E$ is not constant, $E$ is not \Disc{p}{\mathcal{I}}.

Consider any basic $\mathcal{J}$ with $\mathcal{I}\not\leq\mathcal{J}$.  Take any basic $\mathcal{B}_{k,\mathcal{J}}$ set $D=\{\vec x_{[k]}\mid\forall J\in\mathcal{J}'\, \vec x_J\in D_J\}$ where $\mathcal{J}'$ is any permutation of $\mathcal{J}$.  Then there is an $I_0\in\mathcal{I}$ so that no $J\in\mathcal{J}'$ has $I_0\subseteq J$.   Then
\[\int \chi_E\chi_{D}d\mu^k=\int\chi_E\prod_J\chi_{D_J}d\mu^k=\iint \chi_E\prod\chi_{D_J}d\mu^{I_0}d\mu^{[k]\setminus I_0}.\]
Consider some fixed $\vec x_{[k]\setminus I_0}$ and set $E'=E(\vec x_{[k]\setminus I_0})$, $D'=D(\vec x_{[k]\setminus I_0})$ and for each $\pi$, $C'_\pi=C_\pi(\vec x_{[k]\setminus I_0})$.

Fix some $m$ with $1/m<p$ and $1/m<1-p$, and for any $i<m$, let $S^\pi_i=\{\vec x_{I_0}\mid\sum_{I\in\mathcal{I},I\neq I_0}f_I(\vec x_{\pi^{-1}(I)})\in[i/m,(i+1)/m)\}$.  ($\vec x_I$ depends on the fixed choice of $\vec x_{[k]\setminus I_0}$ as well as $\vec x_{I_0}$.)  Note that $D'\cap C'_\pi\cap S^\pi_{i}$ is $\mathcal{B}_{I_0,|I_0|-1}$-measurable.  Let $U^\pi_i=\{\vec x_{I_0}\mid [i/m+f_{I_0}(\vec x_{\pi^{-1}(I_0)})\mod 1,(i+1)/m+f_{I_0}(\vec x_{\pi^{-1}(I_0)}))\subseteq [0,p)\}$ and $V^\pi_i=\{\vec x_{I_0}\mid [i/m+f_{I_0}(\vec x_{\pi^{-1}(I_0)}),(i+1)/m+f_{I_0}(\vec x_{\pi^{-1}(I_0)}))\subseteq [p,1)\}$.  Note that for each $i$, $\mu^{|I_0|}(U^\pi_i\cup V^\pi_i)=1-2/m$, $U^\pi_i$ is \Disc{p-1/m}{|I_0|-1}, and $V^\pi_i$ is \Disc{1-p-1/m}{|I_0|-1}.  Also, for any $i$, $C'_\pi\cap S^\pi_i\cap U^\pi_i\subseteq E'\cap C'_\pi$ while $S^\pi_i\cap V^\pi_i\cap E'\cap C'_\pi=\emptyset$.

Then for any $\pi$,
\begin{align*}
  \mu^{|I_0|}(E'\cap D'\cap C'_\pi)
&=\sum_i\mu^{|I_0|}(E'\cap D'\cap C'_\pi\cap S^\pi_i)\\
&=\sum_i\mu^{|I_0|}(E'\cap D'\cap C'_\pi\cap S^\pi_i\cap U^\pi_i)+\\
&\ \ \ \ \ \ \ \ \mu^{|I_0|}(E'\cap D'\cap C'_\pi\cap S^\pi_i\cap V^\pi_i)+\\
&\ \ \ \ \ \ \ \ \mu^{|I_0|}(E'\cap D'\cap C'_\pi\cap S^\pi_i\setminus(U^\pi_i\cup V^\pi_i))\\
&=\sum_i\mu^{|I_0|}(D'\cap C'_\pi\cap S^\pi_i\cap U^\pi_i)+\mu^{|I_0|}(E'\cap D'\cap C'_\pi\cap S^\pi_i\setminus(U^\pi_i\cup V^\pi_i))\\
&=\sum_i(p-1/m)\mu^{|I_0|}(D'\cap C'_\pi\cap S^\pi_i)+\mu^{|I_0|}(E'\cap D'\cap C'_\pi\cap S^\pi_i\setminus(U^\pi_i\cup V^\pi_i)).\\
\end{align*}
Therefore $\left|\mu^{|I_0|}(E'\cap D'\cap C'_\pi)-p\mu^{|I_0|}(D'\cap C'_\pi)\right|<3/m$.  Since we may make $1/m$ arbitrarily small, $\mu^{|I_0|}(E'\cap D'\cap C_\pi)=p\mu^{|I_0|}(D'\cap C_\pi)$ for each $\pi$, and so $\mu^{|I_0|}(E'\cap D')=p\mu^{|I_0|}(D')$.

Since this holds for each $\vec x_{[k]\setminus I_0}$, $\int \chi_E\chi_{D}d\mu^k=p\mu^k(D)$.  Since this holds for every permutation of every $\mathcal{B}_{k,\mathcal{J}}$-measurable set $D$, $E$ is \Disc{p}{\mathcal{J}}.
\end{proof}

\begin{theorem}
For any $\mathcal{I}$ subset-free on $[k]$ and any $p\in(0,1)$, there is sequence of $k$-uniform hypergraphs $\{(V_n,E_n)\}$ which is not \Disc{p}{\mathcal{I}}, but for any $\mathcal{J}$ such that $\mathcal{I}\not\leq\mathcal{J}$, $\{(V_n,E_n)\}$ is \Disc{p}{\mathcal{J}}.
\end{theorem}
\begin{proof}
  By the previous Theorem, we may take any ultraproduct $(\Omega,\{\mu^k\},\{\mathcal{B}_k\})$ and choose a symmetric $E\subseteq\Omega^k$ which is not \Disc{p}{\mathcal{I}} but is \Disc{p}{\mathcal{J}} whenever $\mathcal{I}\not\leq\mathcal{J}$.  By Theorem \ref{thm:sampling}, we may find a sequence $\{(V_n,E_n)\}$ so that for every $H$, $\lim_{n\rightarrow\infty}t_H((V_n,E_n))=t_H(E)$.  Using Theorem \ref{thm:main_inf} on $E$ and Theorem \ref{thm:main_finite} on the sequence $\{(V_n,E_n)\}$ to count copies of $M_k[\mathcal{I}]$ and $M_k[\mathcal{J}]$, we see that $\{(V_n,E_n)\}$ is not \Disc{p}{\mathcal{I}} but is \Disc{p}{\mathcal{J}} for each $\mathcal{J}$ with $\mathcal{I}\not\leq\mathcal{J}$.
\end{proof}

\section{Conclusion}

There are many notions equivalent to $p$-quasirandomness for graphs, and this paper only discusses a few.  Other equivalent versions of \Disc{p}{1} or \Disc{p}{\mathcal{I}} when $\mathcal{I}$ is a partition are given in \cite{MR2864650} and \cite{2012arXiv1208.4863L,2013arXiv1309.3584L}; it would be interesting to see some of these equivalences generalized to all choices of $\mathcal{I}$.  A particularly interesting case is the spectral characterization given in \cite{2012arXiv1208.4863L,2013arXiv1309.3584L}, since the spectral analog for the graph case is well understood: we can associate to a graph $E$ the operator $f(x)\mapsto \int f(x)\chi_E(x,y) d\mu^1(x)$ mapping $L^2$ functions to $L^2$ functions, and the spectral properties of this operator are the limit of the spectral properties of the graphs $E_n$.

We could take the case $\mathcal{B}_{k,\emptyset}$ to be the trivial $\sigma$-algebra $\{\emptyset,\Omega^k\}$; then saying $E$ is \Disc{p}{\mathcal{I}} amounts to saying that $\mathbb{E}(E\mid\mathcal{B}_{k,\mathcal{I}})=\mathbb{E}(E\mid\mathcal{B}_{k,\emptyset})=p$.  More generally, we could take $\sigma$-algebras corresponding to $\mathcal{I}<\mathcal{J}$ and consider hypergraphs $E$ which are ``relatively random'', in the sense that $\mathbb{E}(E\mid\mathcal{B}_{k,\mathcal{J}})=\mathbb{E}(E\mid\mathcal{B}_{k,\mathcal{I}})$.  For instance, the case where $\mathbb{E}(E\mid\mathcal{B}_{k,\mathcal{I}})=\mathbb{E}(E\mid\mathcal{B}_{k,1})$ is a hypergraph $E$ which has a prescribed set of regularity partitions.  For instance, when $k=2$, it is sometimes useful to consider graphs which are bipartite with a specified density $p$, and random relative to this property.  This is precisely the same as saying that $\mathbb{E}(E\mid\mathcal{B}_{2,1})$ is the function which is equal to $0$ on pairs in the same component and $p$ on pairs which cross components.  We expect that many results about quasirandomness generalize to this setting.

\bibliographystyle{plain}
\bibliography{../../Bibliographies/main}
\end{document}